\documentclass[sn-mathphys,Numbered]{sn-jnl}
\usepackage{graphicx}%
\usepackage{multirow}%
\usepackage{amsmath,amssymb,amsfonts}%
\usepackage{amsthm}%
\usepackage{mathrsfs}%
\usepackage[title]{appendix}%
\usepackage{xcolor}%
\usepackage{textcomp}%
\usepackage{manyfoot}%
\usepackage{booktabs}%
\usepackage{framed}%
\usepackage{comment}%
\usepackage{subcaption}%
\usepackage{mathtools}
\usepackage[utf8]{inputenc}
\usepackage{tikz}
\usepackage{pgfplots}
\usetikzlibrary{calc}
\usetikzlibrary{patterns}
\pgfplotsset{compat=newest}
\usetikzlibrary{arrows.meta} 
\usepackage[linesnumbered,ruled,vlined]{algorithm2e}
\usepackage{listings}%
\usepackage{mathtools}
\usepackage{adjustbox}
\pgfmathdeclarefunction{gauss}{2}{%
  \pgfmathparse{1/(#2*sqrt(2*pi))*exp(-((x-#1)^2)/(2*#2^2))}%
}

\usepackage{xcolor}
\usepackage{import}

\newtheorem{theorem}{Theorem}
\newtheorem{proposition}[theorem]{Proposition}%

\newtheorem{hypothesis}{Hypothesis}
\makeatletter
\newcommand{\sethypotesistag}[1]{
  \let\oldthehypothesis\thehypothesis
  \renewcommand{\thehypothesis}{#1}
  \g@addto@macro\endhypothesis{
    \addtocounter{hypothesis}{-1}
    \global\let\thehypothesis\oldthehypothesis}
  }
\makeatother

\newtheorem{example}{Example}%
\newtheorem{lemma}[theorem]{Lemma}

\newtheorem{remark}{Remark}%

\newtheorem{definition}{Definition}%

\newcommand{\Nte}[1]{  \textcolor{red}{#1}  } 
 
\newcommand{\Njd}[1]{  \textcolor{orange}{#1}  } 

\newcommand{\Dist}{\mathcal{D}}

\newcommand{\pfrac}[2]{ {#1}/{#2}  } 

\def\dint{{\displaystyle \int}}

\renewcommand{\exp}[1]{ \mbox{exp} \left(  #1 \right) }
\renewcommand{\log}[1]{ \mbox{log} \left(  #1 \right) }

\def\E{{\mathbb E}}
\def\Ec{{\mathcal H}}

\def\Nc{{\mathcal N}}

\def\Xc{{\mathcal X}}
\def\Vc{{\mathcal V}}

\def\de{{\delta,\ee}} 
\def\se{{\ee,\sigma}}

\def\Wc{{\mathcal W}}

\def\D{{a}}
\def\V{{b}}

\def\u{{\mathcal U}}
\def\v{{\mathcal D}}

\def\J{{\mathcal J}}

\def\K{{c}}

\def\I{{\mathcal I}}

\def\ee{{\varepsilon}}
\def\etae{{\eta_\ee}}

\def\cqq{{\coloneqq}}

\def\Fc{{\mathcal F}}
\def\Xc{{\mathcal X}}
\def\Hc{{\mathcal H}}
\def\Ic{{\mathcal I}}

\def\S{{\mathcal S}}

\def\Lc{{\mathcal L}}

\def\R{{\mathbb R}}

\def\Pc{{\mathcal P}}
\def\Q{{\mathbb Q}}

\def\tP{{\widetilde{\mathbb P}}}

\def\tV{\widetilde{b}}
\def\tD{\widetilde{a}}
\def\Dd{{a}_\delta}
\def\Vd{{b}_\delta}
\def\Pd{\P_\delta}
\def\Xd{X_{\delta}}

\def\tro{\widetilde{\rho}}

\def\P{{\mathbb P}}
\def\R{{\mathbb R}}
\def\bD{{\overline{a}}}

\def\bs{{\overline{\sigma}}}
\def\bP{{\overline{\mathbb P}}}
\def\tP{{\widetilde{\mathbb P}}}
\def\tJ{{{\mathcal J^P}}}

\def\tm{{\widetilde{m}}}
\def\tn{{\widetilde{n}}}

\def\zi{{x_{i+1}-x_i}}

\def\tP{{\widetilde{\mathbb P}}}

\def\tP{{\widetilde{\mathbb P}}}

\SetKw{Break}{break}
\def\tri{{i\rightarrow i+1}}
\def\m2i{{i,i+1}} 

\begin{document}

\title[Entropic  Semi-Martingale Optimal Transport]{Entropic  Semi-Martingale Optimal Transport}

\author*[1,2]{\fnm{Jean-David} \sur{Benamou}} \email{Jean-David.Benamou@inria.fr}

\author[2,3]{\fnm{Guillaume} \sur{Chazareix}}\email{Guillaume.Chazareix@inria.fr}

\author[1,4]{\fnm{Marc} \sur{Hoffmann}}\email{hoffmann@ceremade.dauphine.fr}

\author[3]{\fnm{Gr\'egoire} \sur{Loeper}}\email{gregoire.loeper@bnpparibas.com}

\author[5]{\fnm{Fran\c{c}ois-Xavier} \sur{Vialard}}\email{francois-xavier.vialard@univ-eiffel.fr}

\affil[1]{\orgdiv{CEREMADE}, \orgname{Universit\'e Paris Dauphine-PSL} 
}

\affil[2]{\orgdiv{MOKAPLAN}, \orgname{INRIA Paris}
}

\affil[3]{\orgdiv{BNP Paribas Global Markets}
}

\affil[4]{\orgdiv{Institut Universitaire de France}}

\affil[5]{\orgdiv{LIGM, Universit\'e Gustave Eiffel}}

\abstract{ Entropic Optimal Transport (EOT), also known as the Schrödinger problem, aims to find random processes with given initial and final marginals, minimizing the relative entropy (RE) with respect to a reference measure.  Both processes (the reference and the controlled one) necessarily share the same diffusion coefficients to  ensure  finiteness of the RE. This initially suggests that controlled-diffusion Semi-Martingale Optimal Transport (SMOT) problems may be incompatible with entropic regularization. However, when the process is observed at discrete times, forming a Markov chain, the RE remains finite even with variable diffusion coefficients. In this case, discrete semi-martingales can emerge as solutions to multi-marginal EOT problems. For smooth semi-martingales, the scaled limit of the relative entropy of their  time discretizations converges to the ``specific relative entropy'',  a convex functional of the variance. This observation leads to an entropy regularized time discretization of the continuous SMOT problems, enabling the computation of discrete approximations via a multi-marginal Sinkhorn algorithm. We prove convergence of the time-discrete entropic problem to the continuous case, provide an implementation, and present numerical experiments supporting the theoretical results.}

\keywords{Semi-Martingale Optimal Transport, Multi-Marginal Optimal Transport, Entropic Penalisation, Specific Relative Entropy, Sinkhorn Algorithm}

\maketitle

{ \small
\subsection*{Notations}

Time is denoted by $t\in [0,1]$ and space by $x\in \R^d$. 

$(\Omega,{\cal F})$ is the space of continuous paths $C([0,1];\R)$ equipped with its Borel sigma-field ${\cal F}$ for the topology of uniform convergence (equivalently, with the sigma-field of all the coordinate mappings)
 and $\Pc(\Omega)$ is the space of probability measures on $(\Omega, {\cal F})$.

The evaluation map at time $t$ is denoted by
 $e_t : \omega \in \Omega \rightarrow \omega_t $, also called the ``position at time $t$'' map. 
 
 The push forward of a measure $\mu$ by a
mapping $T$ is written $T_\# \, \mu$, and defined by $T_\# \, \mu (B) = \mu (T^{-1}(B)) $ for every Borel set $B$.
For $\P \in \Pc(\Omega)$,  $\P_t  = (e_t)_\# \P $ is the  marginal at time $t$, 
and likewise for $(t_1, t_2 , ...)$ marginals: $\P_{t_1, t_2, ...}  = (e_{t_1,t_2,...})_\# \P $. If the time marginal $\P_t$ si absolutely continuous, its densioty with respect to the Lebesgue measure is denoted by $\rho_t$.

Let $(t_i)_{i=0,N}$ be  a regular time discretisation of $[0,1]$ associated with the time step $h := \frac{1}{N}$. 
We use a subscript $.^h$ when dealing with time discrete quantities. For 
instance  
$\P^h  = (e_{t_0,t_1,...,t_N})_\# \P $.
Given $\P^h$ we use the simplified
marginal notation $\P_i^h = (e_{t_i})_\# \P^h$, $\P_\m2i^h = (e_{t_i,t_{i+1}})_\# \P^h$, etc\ldots

We sometimes abuse notation slightly using the same symbol for a measure and 
its density with respect to the Lebesgue measure, i.e.  $x_i \in \R^d \rightarrow \P^h_i(x_i)$ a function while  $\P^h_i(B) = \int_B d\P^h_i(x_i) = \int_B \P^h_i(x_i) \, dx_{i}$ for any Borel 
$B \in \R^d$. The same holds for multivariate measures. 
More generally the subscript $i$ or $i,i+1$ always indicates that the variable is a measure  or function over 
$\R^d_{t_i}$ or $ \R^d_{t_i} \times \R^d_{t_{i+1}}$. 

Throughout the paper, $|\cdot|$ denotes indifferently an Euclidean norm, absolute value for a scalar, Euclidean norm for a vector, or Frobenius norm for a matrix.  

The indicator or characteristic function of a set is be denoted $\iota$, hence
 $ \iota_{[\lambda, \Lambda]}(x) \cqq  0 \text{ if }  x\in [\lambda, \Lambda] \mbox{ and } +\infty \text{ otherwise},  $
and by extension, for a symmetric matrix $a$,
 $ \iota_{[\lambda, \Lambda]}(a) \cqq 0 \text{ if } \lambda Id \leq a \leq \Lambda Id \mbox{ and }  +\infty \text{ otherwise}. $

A standard notation for the transition probability from $(s,x)$  to $(t,y) $ is $\Pi(s,x,t,y) = \P( X_t = y | X_s= x)$  but when dealing with 
time discrete measures we will also  use the  more compact  notation  and decomposition: 
$\P^h_\tri(x_i, dx_{i+1})      \coloneqq  { \P^h_\m2i( x_i, dx_{i+1})}/{\P^h_i( x_i)} $ 
 (or simply $\P^h_\tri$ when no confusion is possible). Hence  for all $x_i$,  $\P^h_\tri(x_i,\cdot)$ is a  probability on $\Pc(\R^d_{t_{i+1}})$ 
while $\P^h_\m2i$ is a joint probability defined on $ \R^d_{t_i} \times \R^d_{t_{i+1}}$. 

The expectation with respect to $\P$ is denoted by
$\E_{\P}(\cdot)$
or $\E(\cdot)$ if there is no ambiguity. 
Finally, we denote by $(\rho_0,\rho_1) \in \Pc(\R^d) \times \Pc(\R^d)  $ two  probability densities with finite second moments, that will play the role of initial and final distributions for the optimal transport problem.
The relative entropy of a measure  $\P$ with respect to $\bP$ is given as:
$
  \Ec(\P | \bP)  \cqq   \E_{\P}  [ \log{ {d\P}/ {d\bP}}  ]  \mbox{ if }  \P \ll \bP  \mbox{ and }  +\infty \quad \text{otherwise}. $



\tableofcontents

\section{Introduction} 
\label{intro}

The distinction between Entropy-Regularized Optimal Transport (EOT) and Semi-Martingale Optimal Transport (SMOT), further discussed in sections  \ref{s11} and \ref{smot}, hinges on the underlying stochastic processes being optimized. EOT \cite{PeyreB,LeonardS,GalichonB,Benamou2021}  is a well-established method in computational optimal transport, equivalent to relative entropy (RE) minimization with respect to a given reference diffusion, typically the Wiener measure. To ensure the finiteness of the RE, both the reference and controlled processes must have the same diffusion coefficients. This characteristic makes EOT particularly suitable for applications where controlling the drift is adequate and computational efficiency is paramount, such as in machine learning, statistics, and physics.

In contrast, SMOT \cite{TT13}  introduces greater complexity by allowing control over both the drift and the diffusion coefficients. This expanded control framework enables one to model more intricate dynamics that cannot be adequately captured by a drift control alone, with model calibration in finance serving as a prime example.

When discretizing continuous diffusions processes in time and formulating them as Markov chains, the RE may remain finite even for different diffusion coefficients that imply orthogonal laws on the the space of continuous trajectories. Notably, as the time discretization step approaches zero, the asymptotic behavior of the discrete RE can be analyzed. Properly scaled by the time step, the first-order term converges to a well-defined limit known in the literature as the specific relative entropy (SRE), which represents a divergence between the respective diffusion coefficients.  This concept was apparently initially  proposed  in \cite{Gantert91}.

In section  \ref{smotprimal} , we use  SRE as a regularization term for the time-continuous SMOT problem 
It can also be interpreted as a fidelity term pertaining to the diffusion coefficient of the optimal associated semimartingale. The natural time discretization 
optimizes over Markov chains, and the SRE becomes a ``correctly'' scaled discrete RE regularization. In this framework, arbitrary time discretizations of semimartingales can be realized as solutions to multi-marginal EOT problems. Our first contribution, Theorem \ref{maintheo}, demonstrates that the sequence of Markov chains indexed by the time step converges weakly (similar to Gamma-convergence) to a minimizer of the time-continuous ESMOT problem (V0). The convergence proof is sketched in section \ref{mainproof}, with detailed elaboration provided in Appendix \ref{annex}.

The numerical solution relies on the dual formulation of the problem, as presented in equations (\ref{DKMMOT}-\ref{DELTA1}-\ref{FD1}). We formally address the duality in section \ref{frd}, where we apply Fenchel-Rockafellar convex duality, common in optimal transport, by using a linear change of variables involving moments rather than conditional moments. The integrand is interpreted as a convex lower semicontinuous perspective function  (\ref{convex}). Our implementation is based on the Sinkhorn algorithm, which is detailed in section  \ref{secsink}. While the convergence analysis of Sinkhorn in both continuous and discrete cases remains to be thoroughly examined, we present the implementation and several test case experiments in section \ref{numerics}.

Our novel regularization approach illustrates that EOT can be extended to more general diffusion processes, bridging the classical EOT framework and broader SMOT problems, thus giving rise to the term Entropic Semi-Martingale Optimal Transport (ESMOT). The significance of this new approach is not solely theoretical. Traditional methods often involve maximizing the dual problem through gradient ascent or primal-dual techniques, which require solving a fully nonlinear Hamilton-Jacobi-Bellman equation at each iteration  \cite{BBA} and \cite{PapaA,LoeperA}. In contrast, our approach using the Multi-Marginal Sinkhorn algorithm, which builds on  \cite{BenamouE} and \cite{BenamouMFG2}, is straightforward to implement and comes with convergence guarantees. This method has been successfully applied to model calibration using derivative price observations  \cite{FinancePaper}, yielding smoother volatility surfaces compared to alternative parametric or PDE-based methods.  Potential applications are probably not limited to finance: for instance, our approach could improve on the recent OT modeling for genes' developmental trajectories of cells \cite{Sbgr}.

There are very few related works in the literature,   \cite{Avellaneda}  highlights the challenges of using relative entropy between singular measures and discusses a discrete version of SRE. However, this work remains entirely at the time discrete level and does not connect with EOT. Additionally, the characterization of specific relative entropy for smooth Itô diffusions can also be found in \cite{backhoff23} and \cite{backhoff2024}, where the authors apply time discretization and scaling techniques, substituting relative entropy with $p$-Wasserstein distance and investigating the properties of this new divergence between martingales.

\section{Background}
\label{Rintro}

\subsection{Entropic Optimal Transport (EOT)}
\label{s11} 
EOT is a well-studied problem.  Given an initial distribution $\rho_0$ on $\mathbb R^d$ and $\bs^2 >0$, call $\bar\P$ the distribution of a Brownian motion with variance $\bs^2 $ and initial random condition distributed according to $\rho_0$. Call 
$\bar\P_{0,1}$ its joint (initial, final) ditribution at times $0$ and $1$. 
The EOT problem can be formulated in a {\it static form} as
\begin{equation} \label{2M}
\inf_{\P_{0,1}\in \Pc(\R^d\times\R^d),  \, \P_0=\rho_0,  \, \P_1=\rho_1}  \bs^2 \,  \Ec(\P_{0,1}  | \bar\P_{0,1}).
\end{equation} 

The  equivalent time continuous {\em dynamic  form} version of (\ref{2M}), known as the Schr\"odinger problem, is
\begin{equation} \label{SB}
  \inf_{  \P \in \Pc  ,   \, \P_{0} = \rho_0,  \, \P_{1} = \rho_1 }   \bs^2 \,  
  \Ec (\P | \bP).
\end{equation} 
The equivalence is easily seen from the additive property of the relative entropy  
applied to the disintegration of the measures $\P$ and $\bP$ 
\[
 \Ec (\P | \bP) =   \Ec (\P_{0,1} | \bP_{0,1} ) +  \E_{ \P_{0,1}} \left( \Ec ( \P(. | X_0 , X_1)  | \bP(. | X_0 , X_1)) \right). 
\]
In particular, the optimal $\P$ and $\bP$ necessarily share the sam point-to-point Brownian bridges.  It is well-known 
(see \cite{LeonardS} for a  survey) that  the minimizers of the Shr\"odinger problem are semi-martingales  with {\em fixed} volatility $\bs$ and 
local drift converging to the OT map as $\bs \to 0$.   There is a large literature  based 
on this property as, unlike the classical OT problem,  the EOT is a strictly convex optimization problem and 
can be solved efficiently through the celebrated Sinkhorn algorithm (see \cite{GalichonB,PeyreB}).    
The Entropic regularisation and Sinkhorn solving approaches are not limited to the static two-marginals problem (\ref{2M}). They can be applied to 
multimarginal measures. Let  $\P$ be the law of a continuous semi-martingale, and $\P^h =(e_{t_0,t_1,...,t_N})_\# \P  $   its time discretization  (and using a similar notation for $\bP$).  
Remark that even tough  $\Ec (\P^h | \bP^h)$ is finite,  a  discete time multi-marginal entropic    $ \bs^2 \, \Ec (\P^h | \bP^h)$ regularization has the strong consequence that in the limit $h\searrow 0$,  $\P^h$ must have  asymptotically the same volatility $\bs$ otherwise the continuous time entropy  $ \Ec (\P | \bP) $ is infinite.
It was used to optimize over constant prescribed diffusion model in  \cite{BenamouMFG2,PeyreGF} in particular. 

\subsection{Semi Martingale OT (SMOT)} 
\label{smot} 
SMOT allows for local in time and space diffusion coefficients.  Our goal is to extend the Entropic regularisation approach in this setting.

We closely follow the framework proposed by \cite{TT13} and also in \cite{GLpathdep}.
On the path space $\Omega = C^0([0,1],\R^d)$, we consider the canonical process $X_t(\omega) = \omega(t) $ and its associated canonical filtration, $(\Omega,{\cal F}, ({{\cal F}_t})_{t \in [0,1]})$.
Let ${\cal P}^0$ be the set of  probability measures $\P$ adapted to ${\cal{F}}$, such that $X$ is an $({\cal F},\P)$-semimartingale on $[0,1]$.
Thus, $X_t$ can be decomposed as
\[
X_t=X_0+B^\P_t+M^\P_t,\quad \langle X\rangle_t=\langle M^\P\rangle_t=A^\P_t,
\]
where $M^\P$ is 
an $({\cal F},\P)$-martingale on $[0,1]$. In addition, $(B^\P,A^\P)$ is ${\cal F}$-adapted and $\P$-$as$ absolutely continuous with respect to time. In particular, $\P$ is said to have characteristics $(a^\P,b^\P)$, which are defined in the following way,
\[
a^\P_t=\frac{d A^\P_t}{dt},\quad b^\P_t=\frac{d B^\P_t}{dt}\, \cdot
\]
Note that $(a^\P,b^\P)$ is ${\cal F}$-adapted and determined $d\P \otimes dt$ almost everywhere. We might use $\D,\V$ instead of $\D^\P,\V^\P$ for simplicity, but it is fundamental to note that the object we are optimising is the measure $\P$, from which $a, b$ are derived.
Letting $\S^d$ be the set of symmetric matrices and $\S^d_+$ the set of positive semidefinite matrices, $(a^\P,b^\P)$ takes values in the space $\S^d_+ \times \R^{d}$.  For $a_1,a_2 \in \S_+^d$, we define $a_1:a_2 := \text{trace}(a_1^\intercal a_2)$.  

Denote by ${\cal P}^1\subset {\cal P}^0$ the set of probability measures $\P$ whose characteristics $(a^\P,b^\P)$ are $\P \otimes dt$-integrable: 
\[
    \E\bigg(\int_0^1 |a^\P| +|b^\P|\, dt\bigg)< \infty.
\]
Recall that $\P_t=X_{t\,\#}\P$, denotes the marginal of $\P$ at time $t$.

The SMOT stochastic control problem
(see  \cite{TT13,GLS,GLW22,GLpathdep}  ) takes the form 
\begin{equation} \label{SDK}
  \inf_{ {\P \in \Pc^1 }} 
\E_{\P} \left( \int_0^1    F(t,X_t,b^\P_t,a^\P_t) \, dt \right)  + \Dist(\P_0,\rho_0) + \Dist(\P_1,\rho_1),
\end{equation}
where  $F$ has to be convex in $(a,b)$ at every $(t,x)$ and also enforces the non-negativity of $a$ the diffusion coefficient,  $\Dist(\P_t,\rho)$ is convex in $\P_t$, for a given $\rho$. Typically $\Dist$  will be a penalty function to enforce that $\P_t$ is either equal or close to $\rho$ at initial and final times. We will use hard constraints as $\iota_{\rho_t}(\P_t)$, as in the usual optimal transport problem in the numerical section. Still, the convergence theorem relies on a soft constraint for instance $\Wc_2(\P_t,\rho_t)$ the quadratic Wasserstein distance. We will assume that  there exists  a positive constant $C$ such that for every pair $\rho_1,\rho_2$ (denoting $\Wc_2(\rho_0,\rho_1)$ the classical quadratic  Wasserstein distance)
\\
\sethypotesistag{H1}
\begin{hypothesis}[H1]\label{H-1}
We assume the following hypotheses on $\Dist$
\begin{enumerate}
    \item $\Dist$ is strictly convex on its domain,
    \item $\forall \rho_1,\rho_2 \in \mathcal{P}(\R^d), \,\text{ such that } \Dist(\rho_1,\rho_2)<+\infty \text{, then }  \int |x|^4 d\rho_1(x) < \infty \,$ ,
    \item $\forall \rho_2 \in \mathcal{P}(\R^d)$, there exists $C(\rho_2) >0$ such that 
   $ |\Dist(\rho_1,\rho_2) - \Dist(\rho_3,\rho_2)|\leq C\Wc_2(\rho_1,\rho_3)\,,$ when all the quantities are finite.
\end{enumerate}
\end{hypothesis}
As a consequence, if $\eta$ is a space convolution kernel with variance $v_\eta$, then 
$
|\Dist(\eta\star\rho_1,\rho_2) - \Dist(\eta\star\rho_1,\rho_2)| \leq C {v_\eta}
$ when $\Dist(\eta\star\rho_1,\rho_2) < +\infty$.
\begin{example}
These constraints are satisfied by the family of discrepancies indexed by $M>0$,
\begin{equation}   \Dist_M(\rho_1,\rho_2) = \iota_{C(M)}(\rho_1) +  \Wc_2(\rho_1,\rho_2)\,,
\end{equation}
where $\iota_{C(M)}$ is the convex indicator function of the convex set $C(M) \coloneqq \{ \rho \in \mathcal{P}(\R^d)\,,\, \int |x|^4 d\rho(x) < M\} $.
\end{example}

\begin{remark}[] 
\begin{itemize} 
\item Note that the discrepancy $\Dist$ can be nonsymmetric. In the first argument, we will consider the marginals of the probability measure on the path space on which the optimisation is done. In the second argument, we consider the data which can be any probability measure.
\item Assuming $\rho_0$ and $\rho_1$ are in convex order, 
time-dependent martingale constrained processes (see for example \cite{HT19}) correspond to \eqref{SDK} with 
$ F= F(a) \text{ if } b \equiv 0, +\infty \text{ otherwise.} $ 
In particular, the Bass Martingale Problem (see \cite{backhoffM} for example)  corresponds  to the choice
 $F=|a-Id|^2$.
\item  The  Schr\"odinger bridge problem (\ref{SB}) corresponds to 
$ F= |b|^2 \text{ if } a \equiv \varepsilon, +\infty \text{ otherwise}, $
for a positive constant $\varepsilon$.
\end{itemize} 
\end{remark}



Usual coercivity assumptions on $F$ are
\begin{align}
F \geq C(1+|a|^p + |b|^p) \text{ for } C>0, p>1.
\end{align}
We will restrict to a strongly coercive case:
\begin{align}
\label{H1} \tag{H2}
F=G+\iota_{[\lambda, \Lambda]}(a)+\iota_{[0,B]}(|b|)
\end{align}
with $\lambda,\Lambda, B$ given positive parameters, and $G$ satisfying 
\begin{align}\label{H2}  \tag{H3} 
\mbox{ $G$ is strictly convex and Lipschitz on $[0,B] \times [\lambda,\Lambda]$.}
\end{align}
The bounds in \eqref{H1} are technical assumptions that allow us to later handle 
the convergence of the discrete entropy to the specific entropy. \\

As explained in \cite{GLW22}, when $F$ is a function of $(t,X,b,a)$ the optimal $a$ and $b$ will automatically be functions of $(t,X)$, even if one does not restrict such local processes in the minimisation procedure. This is due to the convexity of  $F$ which favors local functions.
We state here a general and standard existence result for \eqref{SDK}, that can be found in \cite{GLpathdep}:
\begin{theorem} 
\label{th1}
{ Under conditions (\ref{H-1}, \ref{H1}, \ref{H2}), if there exists an admissible solution (i.e. a probability $\P$  with finite energy in \eqref{SDK})    there exists a minimiser to \eqref{SDK}.}
\end{theorem} 
This result follows from standard arguments on the compactness of any minimizing sequence, and the lower semi-continuity of the energy functional.

\subsection{Time discretisation and specific relative entropy} 
\label{dsre} 

The concept of  specific relative entropy   between diffusion processes seems to go back to \cite{Gantert91}, see also \cite{follmer22} and \cite{backhoff23}.  

Let us consider   $\P \in \Pc^1$ and the reference measure $\bP$ as section \ref{s11}. We denote by $\P^h =(e_{t_1,t_2,...})_\# \P$ 
and $\bP^h =(e_{t_1,t_2,...})_\# \bP$ their time discretisation with 
time mesh $h$ (see the notation section). Then under regularity conditions on the drift and diffusion coefficients specified in 
Proposition \ref{p1} below,  the specific relative entropy between  $\P$ and $\bP$  is  defined as the limit: 
\[
  \S^{\Ec}(\P | \bP)  := \lim_{h\searrow 0 } h\, \Ec(\P^h | \bP^h)  \,.
  \]
Under additional smoothness and boundedness  assumptions on the characteristics 
$\V^\P, \D^\P$, the specific entropy can be characterized  (see proposition \ref{p1}) explicitly: 
  \begin{equation} \label{SRE}
  \S^{\Ec}(\P | \bP) = \lim_{h\searrow 0 } h\, \Ec(\P^h | \bP^h) =
    \tfrac{1}{2}  \E_{\P} \left( \int_0^1  \S^{\Ic}(\D_t(X_t) | \bD))   \, dt\right) .
  \end{equation}
 The integrand  $\S^{\Ic}$, is given (in dimension 1) by 
  \begin{equation} \label{SI}
    \S^{\Ic}( a | \bD)   \coloneqq \frac{a}{\bD} - 1-   \log{\frac{a}{\bD}}.
  \end{equation} 
  In the case of a positive definite symmetric matrix in dimension $d$, the definition becomes
  $$ \S^{\Ic}( a | \bD)   \coloneqq \mathrm{Tr}(\bar a^{-1}(a-\bar a))-   \log{\frac{\mathrm{det} (a)}{\mathrm{det} \bar a}}.$$
 This function is strictly convex with minimum at $\bD$,  a barrier for vanishing $a$ and strictly increasing but just sub-linearly as $ a \rightarrow +\infty$. 
 
Thanks to the discretisation and renormalisation by $h$, we will see that the limit \eqref{SRE} is well defined even when the diffusions matrices $a$ and $\bar a$ differ, in which case the probability measures $\P$ and $\bP$ are singular as probability measures on the space of continuous functions, which entails that the 
relative entropy blows up.  This is easily understood in  the following 
simplified setting: instead of discretizing a time continuous process, directly assume that  $\P^h$ and $\bP^h$ are 
a Markov chains  with  transisitions  following normal distribution with space dependent   
$(\mu_i^h, v_i^h)$ coefficients such that: 
\[
\P^h_\tri  = \Nc( x_i +  h\, \mu_i^h(x_i),    h\, v^h_i(x_i)) \]
and likewise  $\bP^h_\tri  = \Nc( x_i,  h\,  \bD)$. A direct computation 
expresses  the discrete relative entropy as 
 \begin{equation} \label{RE1}
h\,     \Ec(\P^h | \bP^h) =   h\, \Ec(\P^h_0 | \rho_0 ) + \dfrac{h}{2}  \sum_{i=0}^{N-1} \E_{\P^h_i}( 
    \S^{\Ic} (v_i^h(X_i^h) | \bD)  +  \dfrac{h}{ \bD} \,
    \|\mu_i^h(X_i^h) \|^2 )\,,
  \end{equation}
  where $(X_i^h)_{i=0..N}$ is the canonical discrete time process associated to  $\P^h$. 
Assuming convergence as $h\searrow 0 $ of the piecewise constant in time interpolation of the coefficients and 
{\em without the $h$ scaling}  in (\ref{RE1})  the zero-order term in (\ref{RE1})  blows up consistently with the definition of relative entropy between singular continuous-time diffusion processes. 
{\em With the scaling}, we recover Definition \ref{SRE}.   

Considering now  an arbitrary Markov chain  $\P^h$ we define  
  \begin{equation} \label{SMOM}
 \left\{  \begin{array}{l} 
  \V^h_i(x_i)  \cqq \dfrac{1}{h} \mathbb E_{\P^h_\tri} \left[ X^h_{i+1}-x_i \right]  ,  \\[10pt]
 \D^h_i(x_i)    \cqq \dfrac{1}{h} \mathbb E_{\P^h_\tri}   \left[ (X^h_{i+1}-x_i) (X^h_{i+1}-x_i)^\star   \right] . 
 \end{array} 
 \right.
\end{equation}    
{Remark that compared to the previous Gaussian example $b=\mu$ but $a =  v + h b^2$. This choice of characteristics is only equivalent when $h\to 0$, but has the fundamental property that $\P_i^ha_i^h, \P_i^h b_i^h$ are linear quantities with respect to $\P^h$.}
Using this definition we can 
show in particular that the discrete relative entropy controls the discrete 
version of the specific relative entropy, we show (Proposition \ref{p1} ii)) that 
\begin{equation} \label{dualSpec0}
 h\,  \Ec(\P^h | \bP^h) \ge   h\, \Ec(\P^h_0 | \rho_0 ) + \frac{h}{2}  \sum_{i=0}^{N-1} \E_{\P^h_i}( 
    \S^{\Ic}(\D_i^h(X_i^h) | \bD)  \, .
\end{equation}
We gather these results in: 
 \begin{proposition}[Specific Relative Entropy] 
  \label{p1}
  \begin{itemize}
      \item[i)] Let $\P \in \Pc_1$, see (\ref{SDK}), 
      assume that both functions $(t,x) \mapsto \V_t(x)$ and $(t,x) \mapsto \D_t(x)$ are twice continuously differentiable, and that the matrix $\D_t(x)$ is invertible for every $(t,x) \in [0,1] \times \mathbb R^d$. Assume further that $(\V,\D)$ are bounded above and $\D$  below by a positive constant (in the sense of strong ellipticity for $a$). 
      Then   (\ref{SRE}) holds.  
      \item[ii)]  Let $\P^h$ be the law of a discrete-time Markov chain and $(\V_i^h,\D_i^h)$ given by (\ref{SMOM}). 
      We assume further that $ d\P^h = \rho. d\bP^h$ 
      with a continous density $\rho$, then (\ref{dualSpec0}) holds. 
      \end{itemize}
\end{proposition}

\begin{proof} 
See  appendix  \ref{A1}.
\end{proof}

\section{Entropic Semi-Martingale Optimal Transport }
\label{esmot}

\subsection{ESMOT primal formulation and discretization}
\label{smotprimal} 

We introduce the $\S^\Ic$-regularized  time continuous functional
 \begin{equation} \label{defI0}
  {\cal I}^0 (\P) \coloneqq\E_{\P} \left( \int_0^1    F(\V^\P_t(X_t),\D^\P_t(X_t))    + \S^{\Ic}(\D^\P_t(X_t) | \bD)) \,  dt \right)  + \Dist(\P_0,\rho_0) + \Dist(\P_1,\rho_1) 
\end{equation}
where $\sqrt{\bD}$ is a given reference volatility target,  $\V^\P,\D^\P$ are the characteristic coefficients associated to $\P \in \Pc^1$ (see section \ref{smot}). Note that there are, formally at least, no obstacles to considering a local in time and space  coefficient $\bD$.  The fidelity terms $\Dist(\P_t,\rho)$ properties are summarized in 
 \eqref{H-1}. 

The  time continuous SRE  regularisation of (\ref{SDK}) is:
 \begin{equation} \label{SDKe}   
 \tag{$\Vc^0$}
  \inf_{ \displaystyle  \P \in \Pc^1 } 
  \I^0  (\P) \,.
\end{equation}

Penalization/regularisation of the volatility is often necessary to deal with underdetermined problems like volatility calibration of pricing models in finance (see \cite{GLpathdep, HL19,GLS,BCL,Avellaneda}).   
   Here,  the choice of the SRE regularisation is  also  driven by its discrete-time   
formulation, {which will allow to rely on a Sinkhorn-like algorithm for numerical resolution}. Indeed, based on (\ref{SRE}) and Theorem \ref{ThStroockVaradhanExtendedInTime} we will use the  natural time discretisation of  (\ref{defI0}): 
 
\begin{equation}
  \label{cost2}
  \I^h(\P^h)   \coloneqq
   h\, \sum_{i=0}^{N-1}   \E_{\P^h_i}  \left( F( \V_i^h(X_i^h), \D_i^h(X_i^h) )  \right) +  h \, \Ec(\P^h | \bP^h)  + \Dist(\P_0^h,\rho_0) + \Dist(\P_1^h,\rho_1),
\end{equation}
with $(\V_i^h, \D_i^h)$, the discrete drift and quadratic variation increments defined in (\ref{SMOM}).
We thus consider
  \begin{equation} \label{SDKh}\tag{$\Vc^h$}
  \inf_{    \left\{ \begin{array}{l} 
  \P^h \in \Pc( \bigotimes_{i=0}^{N} \R_{t_i})\,   s.t.  \\[10pt] 
    h\, \sum_{i=0}^{N-1} \E_{\P^h_i} \left(   \K^h_i(X_i^h) \right)    \le   M 
    \end{array}  \right.} 
     \I^h (\P^h)\,.
\end{equation}
In this formulation, we have introduced an additional  constraint on a  moment of order larger than $4$ ($\alpha >0$) referred to  abusively  as \emph{kurtosis} (kurtosis is $\alpha = 0$,  see also  (\ref{KURT}): 
 \[  \begin{array}{l}
  \K^h_i(x_i)  \coloneqq   h^{-(2+\alpha)}\mathbb E_{\P^h_\tri} ( \| X_{i+1}-x_i\|^{ 4+ 2\,\alpha} )\,. 
  \end{array}  
\]
This additional constraint seems necessary to apply Theorem \ref{ThStroockVaradhanExtendedInTime} and 
ensures that the sequence  $(\P^h)_h$ converges in $\Pc^1$. Controlling the specific entropy through 
(\ref{dualSpec0}) is not sufficient  to guarantee (as in   \cite{SVbook} Chap.11 or \cite{kushner})  
 that the limit of $(\P^h)_h$   is a jump free   diffusion process. 
 We start with an existence result for the  discretised problem:
 \begin{theorem}[] 
\label{Dex}
Under the hypotheses (\ref{H-1}-\ref{H1}-\ref{H2}) problem  
\eqref{SDKh} has a unique Markovian solution $\P^h$.
\end{theorem} 
\begin{proof}
 Appendix \ref{B3}
\end{proof}

Before studying the convergence of the discretized problem to the continuous one, let us first comment on the difference between the two formulations: there is no explicit bound on the kurtosis of the continuous process. This bound comes from the boundedness of the drift and diffusion coefficients as an immediate consequence of Lemma \ref{lemmaMH},  section \ref{A5}.
Therefore, choosing the constant $M$  large enough in the discretised problem will be sufficient to ensure that the continuous limit of of this constaint is not saturated. 

Our  main result is the convergence of the values of the discretized problems to the continuous one:
\begin{theorem}[]
\label{maintheo} 
Let $\rho_0, \rho_1 \in \Pc_2(\R^d)$ be given.
Under the hypotheses (\ref{H-1}-\ref{H1}-\ref{H2}), let 
$\P^0$ be a minimizer of \eqref{SDKe} and $(\P^h)_h$  a sequence 
of minimizers of \eqref{SDKh}.
Then, 
\begin{equation}
\nonumber
\lim_{h\searrow 0}\I^h(\P^h) = \I^0(\P^0)\,.
\end{equation}

\end{theorem}
\begin{proof}
    See  Section \ref{mainproof}.
\end{proof}

\begin{remark}
\label{RR}
Theorem \ref{maintheo} does not elaborate on the convergence of the minimizing sequences $(\V^h,\D^h,\P^h)$, but in the course of the proof  we construct
from $(\P^h)_h$ a family of smooth measures on Markov chains $(\P^h_\ee)_{h,\ee}$ , depending on a small regularisation 
parameter $\ee$,  such that 
\begin{itemize}
    \item[i)]$\I^h(\P^h_\ee)$ is close to $\I^h(\P^h)$ depending on $\ee$,
    \item[ii)] for all $\ee>0$,  $\P^h_\ee$ converges as $h\to 0$ (in a sense to be defined) to $\P_\ee$, a smooth measure on $\Omega$,
    \item[iii)] $\I^0(\P_\ee)$ is close to $\I^0(\P^0)$ depending on $\ee$.    
\end{itemize}
\end{remark}

\begin{remark} 
We point out that this result is not  a classic $\Gamma$-convergence result ($\I^h \rightarrow \I^0$) as we deal 
with discrete time Markov chains on the $I^h$ side and continuous diffusion processes as arguments of  $\I^0$.  
Our  interest is  in the behavior 
of  numerical solutions to (\ref{SDKh}).
\end{remark}

\begin{remark}
The discrete approximation of (\ref{SDKe}) corresponds to a {\em relative entropic regularised} problem 
and will be solved in Sections \ref{duality} and \ref{numerics} using a Sinkhorn algorithm. 
Note that under stronger convexity assumptions on $F$, we can apply this strategy  
to any stochastic control problem of the type (\ref{SDK}) simply by writing $ F = F -\S^\Ic  +\S^\Ic$ 
applying (\ref{SRE}) only to the last term. 
\end{remark}

\subsection{Convergence of the time discretisation}
\label{mainproof} 
The proof shares similarities with  $\Gamma-$convergence and  relies on well-chosen regularization of the minimizers. Since it  is long, 
 we highlight 
the main steps and associated intermediate lemmas.
Details  are deferred to appendix \ref{annex}.  

We start with three technical lemmas. 
The first lemma concerns the scaling of the drift and diffusion coefficients, which is needed to satisfy the hard constraints in (\ref{H1}) when using regularisations of the minimizers.
\begin{lemma}[Diffusion coefficients rescaling]
\label{Glemma} 
Let $X_t$ have law $\P \in \Pc^1$ with characteristics $\V, \D $ satisfying $\lambda Id < \D < \Lambda Id $ and 
$| \V| < B$.  For all $1 \gg \delta>0$, let
\[
\Xd=\sqrt{\alpha}X +\sqrt{\ee} B(t),
\]
$
\ee = \delta \, \frac{\lambda + \Lambda }{\Lambda - \lambda} $ and $ \alpha = 1 - \frac{2 \,\ee }{ \lambda + \Lambda }$ depending  on $\delta$, and  $B(t)$ an independent standard Brownian motion.
Then, $\Xd$ has law $\Pd \in \Pc^1$ with characteristics $\Vd, \Dd$, and we have that:

\begin{itemize}
\item[i)] Letting  
$T_{\sqrt\alpha}(x) = \sqrt\alpha x$, and $\P_{t,\delta} := (e_{t})_\# \Pd$, 
\[ \P_{t,\delta} :=  {\cal N}(0,\ee t) \star [T_{\sqrt\alpha}]_\sharp \, \P_t, \] 
\item[ii)] Scaling of the coefficients:
$
(\lambda +\delta) \, Id < \Dd <  (\Lambda - \delta)\, Id, \mbox{   and   }  |\Vd | \le (1-\delta)B \,. 
$ 
\item[iii)] $\Pd$ narrowly converges to $\P$ as $\delta  \searrow 0$ and 
$\I^0(\Pd)  \le \I^0(\P) + O(\delta)\,,$
where the notation $O(\delta)$ hides a constant that depends on the Lipschitz constant of $G$ and $\lambda, \Lambda, B$.
\end{itemize}    

Moreover, the same holds for a discrete process $X^h$ with law $\P^h$, considering $\P^h_\delta$ the probability of   $X^h_\delta(t) = \sqrt{\alpha} X^h(t) +  \sqrt{\varepsilon} B(t)$ as above.
We have, for $\alpha$ sufficiently close to $1$,
\begin{itemize}
\item[iv)] Scaling of the coefficients:
$
(\lambda +\delta) \, Id < \Dd^h <  (\Lambda - \delta)\, Id, \mbox{   and   }  |\Vd^h | \le (1-\delta)B \,.
$
\item[v)] $\Pd^h$ narrowly converges to $\P^h$ as $\delta  \searrow 0$ and 
\begin{equation}\label{ThDiscreteRescaling}
\I^h( \P^h_\delta)
\leq \I^h( \P^h)  + O(\delta)\,.
\end{equation}

\end{itemize}    
\end{lemma}
\begin{proof}
 Appendix \ref{B1}.
\end{proof}

The next lemma concerns the regularization technique for the time-continuous formulation.

\begin{lemma}[Time (and space) regularization - the continuous case]
\label{ThRegInContinuousTime}
    Let $(\P^0,\V^0,\D^0)$ representing a Markov process $X(t)$ solving, for a Brownian motion $B(t)$,
    \begin{equation}
        dX(t) = \V^0(t,X)dt + \D^0(t,X) dB(t)
    \end{equation} such that $\mathcal{I}(\P^0) < +\infty$. Fix $\sigma > 0$ and let $\varepsilon> 0$ such that $2 \varepsilon \Lambda < \sigma$.
We further assume that $ | b^0|\leq M /(1 + 2\varepsilon)$ and $  a^0 \leq \frac{\Lambda}{(1 + 2\varepsilon)^2}\operatorname{Id}$.
    Then, there exists another Markov process $X_{\varepsilon,\sigma}(t)$, solving the SDE as above  for the quantities $\V_{\varepsilon,\sigma}(t,x),\D_{\varepsilon,\sigma}(t,x)$ which are smooth in $t,x$ such that 
    $
        \mathcal{I}(\P_{\varepsilon,\sigma}) \leq \mathcal{I}(\P^0) + O(\varepsilon) + O(\sigma)\,.
    $
\end{lemma}

\begin{proof}
    Appendix \ref{B21}.
\end{proof}

\noindent
\textbf{Main steps of the proof: }
Up to step 4, we prove that the value of the continuous problem is lower bounded by the $\limsup$ of the values of the discrete problems. For that, we consider a probability with finite value for the continuous problem, rescale and regularize it, and then discretize it to obtain a candidate for the discrete problem.

\begin{enumerate}

\item[{\bf{ * Step 0: }}] Recall that (\ref{SDKh}) is well posed and $\P^h$ is markovian (see Theorem \ref{Dex}). We also have well-posedness  and existence of a solution $\P^0$ of the time continuous problem (\ref{SDKe})  is known (see Theorem 
\ref{th1}).

\item[{\bf * Step 1: }]  Given $\P^0 \in \Pc^1$ a minimizer of (\ref{SDKe}) 
we construct a regularised version $\P^0_{\se}$ by  applying  lemma 
 \ref{ThRegInContinuousTime} ($\sigma$ depends on $\delta$) . We get 
$
\I^0(\P_\se)  \le \I^0(\P) + O(\sigma) + O(\ee).
$


\item[{\bf * Step 2: }]  Set $\P^h_{\se} = (e_{t_0,t_1,...,t_N})_\# \P_{\se}$ 
the time discretisation of $\P_{\se}$, the output of {\bf step 1}. We want to compare $\I^h(\P^h_{\se})$ and 
$\I^0(\P_{\se})$. 

\begin{lemma}[]
\label{LF3}
For all  $h$ sufficiently small w.r.t. $\de$, we have
 \begin{itemize}  
 \item[i)] The family $(\P^h_\de)_h$ satisfies the hypothesis of proposition \ref{p1} $i)$ and in particular (see  (\ref{MarcA})) we have: 
 \[ 
 h\, \Ec(\P^h_\se | \bP^h) =  
 \frac{1}{2}\sum_{i = 1}^N h \, \mathbb E_{\mathbb P_\de}  \big[\mathcal \S^{\Ic}(\D^{\P_\se}(ih,X_{ih})|\bD \big]+O(h^{1/4}) \,.
 \] 
  \item[ii)] It holds $
 \I^h(\P^h_\se)  \le   \I^0(\P_\se) + O(h^\frac{1}{4}) \,.$
 \end{itemize}
\end{lemma}
\begin{proof}
 Appendix \ref{B4}
\end{proof}

\item[{\bf * Step 3: }] For any  minimizer of 
 \eqref{SDKh}, denoted $\P^h$,  
and gathering the results in steps 1 and 2 we get 
\[
\I^h(\P^h) \le \I^h(\P^h_\se) \le \I^0(\P^0) + O(\sigma) + O(\ee) +  O(h^\frac{1}{4})
\,.\]
We pass to the limit first in $h$ and then in $\se$ to obtain 
$  \limsup_{h \searrow 0}   \I^h(\P^h)  \le  \I^0(\P^0) \,.$
The next steps show that the $\liminf$ of the values of the discrete problems is lower bounded by the value of the continuous problem.

\vspace{0.2cm}

\item[{\bf * Step 4: }]   We need the following 
definitions:

For $\P^h \in \Pc( \bigotimes_{i=0}^{N} \R_{t_i}) $
Markovian, we define piecewise constant in time interpolants 
\begin{equation}
\label{cv1}
\tro^h_t = \sum_{i=0}^N   \mathbf 1_{[t_i,t_{i+1})} \P_{i}^h, \quad 
\tV^h_t = \sum_{i=0}^N   \mathbf 1_{[t_i,t_{i+1})} \V_{i}^h, \quad 
\tD^h_t = \sum_{i=0}^N   \mathbf 1_{[t_i,t_{i+1})} \D_{i}^h 
\end{equation}
$( \V_{i}^h, \D_{i}^h)$ defined in (\ref{cSMOM}). We use the moment 
notations 
\begin{equation}
\label{cv2}
\tm^h_t = \tro^h_t\, \tD^h_t,  \quad \tn^h_t = \tro^h_t\, \tV^h_t.
\end{equation} 
We define the proxy functional for $\I^0$
\begin{equation}
    \J(\rho,m,n) = \int_{t,x} F(m/\rho,n/\rho) + \S^{\Ic}(n/\rho)   \, d\rho_t(x) dt\,.
\end{equation} 
 In particular we have for  $\P \in \Pc^1$ of finite  $\I^0$ energy: 
$
\I^0(\P) = \J( \P_t,  \P_t \, \V^\P, \P_t \, \D^\P)\, , 
$ 
and using  Proposition \ref{p1} ii) : 
$
  \J(\tro^h,\tm^h,\tn^h) \le   \I^{h}(\P^h) 
$
We need the following lemma:
\begin{lemma}[Time (and space) regularization - the discrete case]
\label{FXlemma}
    
Let  $(\P^h)_h $ be  sequence of Markov Chains such that $ \sup_{h} \I^h(\P^h) < + \infty$ and such that the fourth moment of $\P^h_0$ is uniformly bounded in $h$. Let $\de >0$. Then, there  exist $\P^h_{\de} \in \mathcal{P}(\R^{dN})$  and $\P_{\de} \in \Pc^1$ such that, up to a subsequence in $h$, 
\begin{itemize}
\item[i)] 
we have
\begin{equation}
\I^h( \P^h_{\de})
\leq \I^h( \P^h)  + O(\varepsilon) + O(h) + O(\delta)\,
\end{equation}
and 
\begin{equation}\label{IneqRiemannSum}
  \I^{h}(\P_\de^h) \geq  \mathcal{J}(\tro_\de^h,\tm_\de^h,\tn_\de^h) \,.
\end{equation}
where $(\tro_\de^h,\tm_\de^h,\tn_\de^h)$ is the change of variable (\ref{cv1}-\ref{cv2}) associated to $\P_\de^h$
\item[ii)] 
$\tP^h_{\de}$ narrowly converges to $\P_{\de} \in \Pc^1$ with characteristics 
$\D_t^{\P_\de} , \V_t^{\P_\de}$.

\item[iii)] 
In addition, we have
\[
\I^0( \P_\de) =   \J(\P_{\de,t},  \tP_{\de,t} \V_t^{\tP_\de} , \tP_{\de,t} \D_t^{\tP_\de}    )   \le \liminf_{h\searrow 0}   \J(\tro^h_\de,\tm^h_\de,\tn^h_\de)\,  .
\]
 \end{itemize}
\end{lemma}

\begin{proof}
 Appendix \ref{B22}.  Note that it is possible to apply this lemma due to Hypothesis \eqref{H-1} which ensures that the initial density's fourth moment is finite.
\end{proof}


\item[{\bf * Step 5: }] Lemma \ref{FXlemma} $i)$ and $ii)$  gives
$
\J(\tro^h_\de,\tm^h_\de,\tn^h_\de)   \le   \I^h( \P^h)    + O(\varepsilon) + O(\delta)  + O(h) 
$
and passing to the limit in $h$  using 
 Lemma \ref{FXlemma} $iii)$ and for a minimizer $\P^0$ of (\ref{SDKe}) we have 
\begin{equation}
\label{F11} 
    \I^0(\P^0) \le   \I^0( \P_\de)  \le   \liminf_{h\searrow 0}  \I^h( \P^h)  + O(\ee) + O(\delta) \,.
\end{equation}

\item[{\bf * Last Step:  }] Applying {\bf Step 5}  to a sequence of minimizers $(\P^h)_h$ of (\ref{SDKh}) 
and using \textbf{Step 3}, we get 
$
\I^0(\P^0) \le \lim_{\de \searrow 0}  \I^0(\P_\de) \le  \liminf_{h\searrow 0}  \I^h( \P^h) \le \limsup _{h\searrow 0}  \I^h( \P^h) \le \I^0(\P^0) 
$
which concludes the proof of Theorem \ref{maintheo} $i)$.
\end{enumerate}

\section{Dual formulation and Sinkhorn algorithm }
\label{duality}

In this section the support of $\P^h$ is  a finite product of compact subsets of $\R^d$:  $\Omega^h \cqq \bigotimes_{i=0}^{N} \Xc_i $. This restriction is consistent with the space discretisation and truncation of our implementation  (Section \ref{numerics}). 
In order to simplify the presentation and the notations we restrict to dimension $d=1$ and consider  a dependence of the payoff $F$ just on the drift $\V$,  generalisations to dependence in  $\D$  follow the same lines.  
We also drop all additional constraints on the coefficients and the 
Kurtosis constraints since it does not seems necessary to obtain 
convergence of the Sinkhorn algorithm in the next section and 
also to let $h$ go to 0 together with the space discretisation step. 
We present the duality without rigorous proofs that are left for further studies. 
Let us just mentioned that Multi-Marginal OT on a product of compact space 
is usually easier to deal with as in \cite{DimarinoG} or \cite{CarlierMM} for instance.  

 The simplified  discrete primal problem becomes: 
 \begin{equation} \label{primal}
  \inf_{  \begin{array}{l}   
   \P^h \in \Pc( \Omega^h)\,   
    \end{array}    }
     \E_{\P^h}  \left(    h\, \sum_{i=0}^{N-1} F( \V_i^h(X_i^h) ) \right)  
     + \Dist(\P^h_0,{\rho_0})   + \Dist({\P^h_1,\rho_1}) +  
     h \, \Ec(\P^h | \bP^h)  
\end{equation}
where  $\V^h_i(x_i)  \cqq \dfrac{1}{h} E_{\P^h_\tri} \left( X^h_{i+1}-x_i \right) $ 
is given as a function of $\P^h$ as in (\ref{SMOM}). 
 We also recall that the reference measure is defined by 
$\bP^h_\tri = \Nc(x_i+ h \V_i^h(x_i) , h \bD ) $ and $ \bP^h_0 = \rho_0$.

\subsection{Fenchel-Rockafellar duality} 
\label{frd} 

We remark that the drifts $\V_i^h$  can be written using 
local linear functions of $\P^h$ (abusing again notations and using $\P^h$
for the distribution and their densities):
\begin{equation} \nonumber
      \V_i^h(x_i) =  \dfrac{1}{h} \, \dfrac{ \E_{\P^h_{\m2i} (x_i,.)} (X^h_{i+1} - x_i) }{ \P_i^h(x_i)} 
\end{equation}
where  $\P^h_{\m2i} (x_i,.)  := \P^h_i \, \P^h_\tri$ is to be understood as the  measure on $\R_{t_{i+1}}$ obtained by freezing the first variable $x_i$ in the joint $\R_{t_i} \times \R_{t_{i+1}}$ probability  $\P^h_{\m2i} $.

The primal problem (\ref{primal}) can be rewritten 
\begin{equation} \nonumber
\inf_{   \displaystyle \left\{   
   \P^h \in \Pc( \Omega^h) \right\}  }
     \Fc(\Delta^\dag \P^h )  + \, h \, \Ec(\P^h | \bP^h)
\end{equation}
using the linear change of variable: 
\begin{equation} 
 \nonumber 
 \begin{array}{rl}
  \Delta^\dag \P^h   \cqq  & \{ ( \P_i^h)_{i=0}^{N} ,
 ( x_i \rightarrow  \dfrac{1}{h} \, \E_{\P^h_{\m2i} (x_i,.)} (X^h_{i+1} - x_i)    )_{i=0}^{N-1}  \} ,  
 \end{array}
 \end{equation} 
and modified functional: 
\begin{equation} 
 \nonumber 
 \begin{array}{rl}
 \Fc( (m_{0,i})_{i=0}^{N} ,
 (m_{1,i})_{i=0}^{N-1} ) \cqq   & \Dist({m_{0,0}},{\rho_0})   + \Dist({m_{0,N}},{\rho_1})   +  h\, \sum_{i=0}^{N-1}  \E_{m_{0,i}}(F( \dfrac{m_{1,i}}{ m_{0,i}} )) \,.
 \end{array}
 \end{equation}

 In this section (again for simplicity) we use hard constraints for the loss function on  initial/final marginals. It corresponds to 
 the characteristic function $\Dist(\mu,\rho) := 0$ if $ \mu = \rho$ and $+\infty$ otherwise.   The (jointly convex)  perspective functions $ (m_0,m_1) \rightarrow  h \, \E_{m_{0}}(F( \pfrac{m_{1}}{ m_{0}} )) $ are naturally  extended to $+\infty$ if $m_0 \le 0 $ and $0$ if $ m_1= m_0 = 0$.  The $ .^\dag$ notation denotes that  $ \Delta^\dag$ is the adjoint of 
the linear operator  (\ref{DELTA00}).  

 Fenchel Rockafellar duality yields the  equivalent dual problem:  
\begin{equation} \label{DKMMOT} 
\sup_{ \displaystyle \Phi^h := (\phi^h_{0,i})_{i=0}^{N} ,
 (\phi^h_{1,i})_{i=0}^{N-1} }
- \Fc^*( - \Phi) - h \, \E_{\bP^h}( \exp{\dfrac{\Delta \,   \Phi^h }{h} }  - 1 )\, . 
\end{equation}
The dual variable is a 
vector of continuous functions on the spaces $(\Xc_i)_{i = 0}^N$ $\Phi^h = (\phi^h_{0,i})_{i=0}^{N} ,
 (\phi^h_{1,i})_{i=0}^{N-1}$. 
The second term in (\ref{DKMMOT}) is the Legendre Fenchel transform  of the relative entropy and 
  \begin{equation} 
 \label{DELTA00} 
   \Delta \, \Phi    =    \bigoplus_{i=0}^{N-1}  {   \phi}_{0,i}  + \dfrac{1}{h} ( \zi ) \,  { \phi}_{1,i}  +   \phi_{0,N}\,.
\end{equation}
The Legendre Fenchel transform 
$\Fc^*$  of $\Fc$ is  explicitly given, using the separability in space and in ``times'' $i$ by: 
  \begin{equation} 
 \nonumber
 \Fc^*( \Phi) = \E_{\rho_0}(   \phi_{0,0}  + h\, F^*( \dfrac{\phi_{1,0}}{h}  )) +  \sum_{i=1}^{N-1} \chi_0(   \phi_{0,i}  + h\, F^*( \dfrac{\phi_{1,i} }{h} )) + \E_{\rho_1}( \phi_{0,N})
 \end{equation}
 where $\chi_0(f):=0$ if $ f$ is the null function in $C(\Xc) $ and $+\infty$ else.
It is convenient to eliminate the $(\phi_{0,i})_{i=1}^{N-1}$  variables from $\Phi$  (we keep the same notations) and simplify  problem  (\ref{DKMMOT}) using 
 \begin{equation} 
 \label{DELTA1} 
   \Delta \, \Phi    =     \phi_{0,0} +  \dfrac{1}{h}   (x_1-x_0)    \,  { \phi}_{1,0}    + \bigoplus_{i=1}^{N-1}  { h\,  F^*( - \dfrac{1}{h} \phi_{1,i} )  +  \dfrac{1}{h}   ( \zi)    \,  { \phi}_{1,i}  }   + \phi_{0,N}
\end{equation}
and 
  \begin{equation} 
 \label{FD1} 
 \Fc^*( \Phi) = \E_{\rho_0}(   \phi_{0,0}  + h\, F^*( \dfrac{\phi_{1,0}}{h}  )) + \E_{\rho_1}(  \phi_{0,N})\,.
\end{equation}
Applying classical Fenchel-Rockafellar duality (\cite{CarlierB} theorem 6.3)
\begin{proposition}[Solutions of \eqref{DKMMOT}]
  \label{pro21}
  \begin{itemize}

    \item[i)] Problem \eqref{DKMMOT} has a unique {\em Markovian} solution which can
      be expressed, in terms of   $\Phi^h = (\phi^h_{0,i=0} ,\phi^h_{0,i=N}), (\phi^h_{1,i})_{i=0}^{N-1}$ a 
      vector of continuous functions on the spaces $(\R_{t_i})$, the solution of the dual problem  \eqref{DKMMOT}-\eqref{DELTA1}-\eqref{FD1}:     
      \[
      \P^h  = \operatorname{exp}{ \dfrac{1}{h} \Delta\, \Phi^h } \, \bP  \,, 
      \]
\[    \Delta \, \Phi    =     \phi_{0,0} +  \dfrac{1}{h}   (x_1-x_0)    \,  { \phi}_{1,0}    + \bigoplus_{i=1}^{N-1}  { h\,  F^*( - \dfrac{1}{h} \phi_{1,i} )  +  \dfrac{1}{h}   ( \zi)    \,  { \phi}_{1,i}  }   + \phi_{0,N} \,.
      \]
    \item[ii)] $\P^h$   satisfies, for all $i$,  the factorisation
$\P^h$ can be tensorized with densities: 
\[
 \P^h  = \rho_0 \,  \Pi_{i=0}^{N-1} \P_{\tri}^h  \quad \quad  \P_{\tri}^h  = \dfrac{\P^h_{\m2i}}{\P^h_i} 
\]
and 
$\forall  \, i$
\begin{equation} 
\label{Kern}
\P^h_{\m2i}    =  \u_{i} \, \exp{ \dfrac{1}{h^2}   ( \zi)    \,  { \phi}_{1,i} }   \,  \bP^h_{\m2i}  \,  \v_{i+1} \,.
\end{equation} 
Each $(\u_i)$ and $(\v_i)$ are  functions defined recursively  $\u pward$ and $\v ownward$ using $ \Phi^h$ and $\bP^h$ by
\begin{equation}
\label{SNOTu}
    \left\{
 \begin{array}{l}
   \u_0 =  \exp{\dfrac{1}{h} \,\phi_{0,0}}  , \\[8pt]
   \u_{i+1} =  \exp{  F^*( - \dfrac{1}{h} \phi_{1,i+1} )}  \, \dint \u_i  \,  \exp{   \dfrac{1}{h^2} ( \zi)    \,  { \phi}_{1,i} }  \bP^h_{\m2i} \, d{x_i} \,  , \\[8pt]
  \mbox{ for  $i =0,N-2$}   .
   \end{array}
    \right.
\end{equation}
\begin{equation}
\label{SNOTv}
    \left\{
 \begin{array}{l}
   \v_N= \exp{\dfrac{1}{h} \, \phi_{0,N} } , \\[8pt] 
   \v_{i} =  \exp{ F^*( - \dfrac{1}{h} \phi_{1,i} ) } \, \dint \bP^h_{\m2i} \, \exp{ \dfrac{1}{h^2}   ( \zi)    \,  { \phi}_{1,i}} \,    \v_{i+1} \, dx_{i+1}  , \\[8pt] \mbox{for  $i =N-1,1$}  .
   \end{array}
    \right.
\end{equation}
{\em Please recall    that all variable subscripted by $i$ are to be understood as function of $x_i$}. 
\end{itemize} 
\end{proposition}

\begin{remark}

The modifications to (\ref{DELTA1}) and (\ref{FD1}) in the 
more general case of time and diffusion dependent functions 
$(t_i,\V_i,\D_i) \rightarrow F_i(\V_i,\D_i)$  are simply 
$\Phi^h = (\phi^h_{0,i=0} ,\phi^h_{0,i=N}), (\phi^h_{1,i})_{i=0}^{N-1} , (\phi^h_{2,i})_{i=0}^{N-1}$ the additional potentials $\phi^h_{2,i}$ are dual to 
$\P^h_i\,\D^h_i$  and 

\begin{align} 
 \nonumber
   \Delta \, \Phi   =  &       \phi_{0,0} +  \dfrac{1}{h}   (x_1-x_0)    \,  { \phi}_{1,0}  
    +  \dfrac{1}{h}   (x_1-x_0)^2    \,  { \phi}_{2,0}   +  \\ \nonumber
    &  \bigoplus_{i=1}^{N-1}  { h\,  F_i^*( - \dfrac{1}{h} \phi_{1,i} , - \dfrac{1}{h} \phi_{2,i} )  +  \dfrac{1}{h}   ( \zi)    \,  { \phi}_{1,i}  }   + \dfrac{1}{h}   ( \zi)^2    \,  { \phi}_{2,i}    \\ \nonumber  &  + \phi_{0,N}
\end{align}
and \begin{equation} \nonumber
 \Fc^*( \Phi) = \Dist^*_{\rho_0}(   \phi_{0,0}  + h\, F_0^*( \dfrac{\phi_{1,0}}{h} , \dfrac{\phi_{2,0}}{h} )) + \Dist^*_{\rho_1}(  \phi_{0,N})\,
\end{equation}
where
$\Dist^*_{\rho_i}$ is the Legendre Fenchel Transform of the (convex)  marginal 
fidelity term $\rho \rightarrow \Dist(\rho,\rho_i)$.

\end{remark}

\subsection{Multi-Marginal Sinkhorn Algorithm}
\label{secsink} 
We are now working with a fixed $h$ and drop the dependence in the notations for clarity. 

The simplest interpretation of Sinkhorn Algorithm  (see \cite{PeyreB}) is to perform an iterative 
coordinate-wise (in the components of  $\Phi$) ascent to maximize 
 the concave 
dual problem (\ref{DKMMOT}-\ref{DELTA1}-\ref{FD1}): 
\[
\sup_{ \displaystyle \Phi := (\phi_{0,0},
 (\phi_{1,i})_{i=0}^{N-1}, \phi_{0,N} )  }
- \Fc^*( - \Phi) - h \, \E_{\bP}( \exp{\dfrac{1}{h} \Delta \,   \Phi }  - 1 ) . 
\]
 One cycle of dual variable optimisation will be indexed by $k$ and the potentials 
 updated ``\`a la Gauss Seidel'' in the inner loop 
over the $N+2$ functions (Kantorovich potentials) in $\Phi$ (the order is not  important).
Notations are a little  bit more involved than in  the classical two marginals problem, to ease the presentation 
we set: 
\begin{equation}
\label{UP}
\left\{
\begin{array}{l}
\Phi^{k,0} = (\phi_{0,0},
 (\phi^k_{1,i})_{i=0}^{N-1}, \phi^k_{0,N} )  , \\[10pt]
\Phi^{k,i+1} = (\phi^{k+1}_{0,0},
 (\phi^{k+1}_{1,0} \ldots \phi^{k+1}_{1,i-1}, \phi_{1,i}, \phi^k_{1,i+1} \ldots \phi^k_{1,N}) , \phi^k_{0,N} )  , \quad  i= 0\ldots N-1 ,\\[10pt]
 \Phi^{k,N+2} = (\phi^{k+1}_{0,0},
 (\phi^{k+1}_{1,i})_{i=0}^{N-1}, \phi_{0,N} ) .
 \end{array} 
 \right.
\end{equation}
One Sinkhorn cycle $k \rightarrow k+1$, the updates from $\Phi^{k,0}$ to $
\Phi^{k+1,0}$  can be written in compact form as a loop on its components:  
\begin{equation} 
\label{SIA}
\left\{
\begin{array}{l}
\phi_{0,0}^{k+1}  = \arg\sup_{\phi_{0,0}} 
- \Fc^*( - \Phi^{k,0}) - h \, \E_{\bP}( \exp{\dfrac{1}{h} \Delta \,   \Phi^{k,0} }    - 1 ) , \\[10pt]
  \phi_{1,j+1}^{k+1}  = \arg\sup_{\phi_{1,j}}  
- \Fc^*( - \Phi^{k,i+1}) - h \, \E_{\bP}( \exp{\dfrac{1}{h} \Delta \,   \Phi^{k,i+1}  }    - 1 ) , \quad  i= 0\ldots N-1 ,\\[10pt]
\phi_{0,N}^{k+1}  = \arg\sup_{\phi_{0,N}} 
- \Fc^*( - \Phi^{k,N+1}) - h \, \E_{\bP}( \exp{\dfrac{1}{h} \Delta \,   \Phi^{k,N+2}  }    - 1 ).
 \end{array} 
 \right.
\end{equation} 
Each of these maximization problems is stricly concave and 
sufficiently smooth (depending on $F^\ast$). They are also separable 
in space and (\ref{SIA}) amounts to solve in sequence 
the following set of equations ($\partial F^*$ is the Frechet derivative, a gradient in finite dimention):
\begin{equation}
  \nonumber
      \left\{
         \begin{array}{l}
 \rho_0 =  \exp{\dfrac{1}{h }\phi_{0,0}^{k+1} (x_0) } \dint \exp{ \dfrac{1}{h^2}   ( x_1-x_0)    \,  { \phi}^k_{1,0} }   \,  \bP^h_{0,1}  \,  \v_{1}^k  \, dx_1    \\[10pt]
 \rho_0 \, \partial F^*( -\dfrac{1}{h }\phi^{k+1}_{1,0}) =  \exp{\dfrac{1}{h }\phi_{0,0}^{k+1} } \dint \dfrac{1}{h }(x_1- x_0 )  \,\exp{ \dfrac{1}{h^2}   ( x_1 - x_0 )    \,  { \phi}^{k+1}_{1,0} }   \,  \bP^h_{0,1}  \,  \v_{1}^k  \, dx_1    \\[10pt]
0 = \u_i^{k+1} \, \dint  \left( \partial F^*( -\dfrac{1}{h }\phi^{k+1}_{1,i}) + \dfrac{1}{h }(\zi)  \right) \,\exp{ F^*( -\dfrac{1}{h }\phi^{k+1}_{1,i}) + \dfrac{1}{h^2}   ( \zi)    \,  { \phi}^{k+1}_{1,0} }   \\[10pt]
 \quad \quad \quad \quad \quad \quad    \quad \quad \quad  \, \ldots  \bP^h_{\m2i}  \,  \v_{i+1}^k  \, dx_{i+1}   \quad \quad \quad \mbox{  for all $i=1,N-1$} \\[10pt]
    \rho_1  =   \u_{N-1}^{k+1} \,  \exp{ \dfrac{1}{h } \phi_{0,N}^{k+1}}    
      \end{array}
      \right. 
  \end{equation}
{Please recall  that all variable subscripted by $i$ are to be understood as function of $x_i$. 
 Each line is to be understood as point-wise in space.} 
The functions $(\u_i^k,\v_i^k)$  are defined by the recursions (\ref{SNOTu}-\ref{SNOTv}) applied  to the iterative Gauss-Seidel update of the potentials: 
(\ref{UP})
\begin{equation}
\label{SNOTuk}
    \left\{
 \begin{array}{l}
   \u^k_0 =  \exp{\dfrac{1}{h} \,\phi^k_{0,0}}   \\[8pt]
   \u^k_{i+1} =  \exp{  F^*( - \dfrac{1}{h} \phi^k_{1,i+1} )} \, \dint \u^k_i   \exp{   \dfrac{1}{h^2} ( \zi)    \,  { \phi}^k_{1,i} }  \bP^h_{\m2i}  \, dx_{i} \,    \\[8pt]
  \mbox{ for  $i =0,N-2$}   
   \end{array}
    \right.
\end{equation}
\begin{equation}
\label{SNOTvk}
    \left\{
 \begin{array}{l}
   \v_N^k= \exp{\dfrac{1}{h} \, \phi^k_{0,N} }  \\[8pt] 
   \v^k_{i} =  \exp{ F^*( - \dfrac{1}{h} \phi^k_{1,i} ) } \, \dint \bP^h_{\m2i} \, \exp{ \dfrac{1}{h^2}   ( \zi)    \,  { \phi}^k_{1,i}}     \v^k_{i+1} \, dx_{i+1}   \\[8pt] \mbox{for  $i =N-1,1$} 
   \end{array}
    \right.
\end{equation}
 The numerical study in Section \ref{numerics} gives experimental convergence curves. In the case of finite dimension in space, i.e,
$\P^h$ as support on a truncated grid in space. The analysis of 
 the convergence (in $k$) of this algorithm  
should rely on a multi-marginal extension of \cite{CominettiA}, see also \cite{CarlierMM} or \cite{DimarinoG} and the references therein.  Note however that they apply to  ``standard'' multi-marginal constraints while we use  a convex payoff of the moments of $\P^h$.
This problem is closer in spirit to the ``weak'' OT \cite{weak} setting or rather the  entropic regularization of 
a weak OT problem.  

\subsection{Numerical experiments} 
\label{numerics}


In all experiments, we use the function $F = \gamma \|\cdot\|^2$ with a large parameter $\gamma$ (typically $1.E06$), this is to enforce a soft martingale constraint (no drift).
The space discretisation is chosen as 
$ dx = (10\, \bs)/K \, \sqrt{h} $ the parameter $K$ corresponds to the number of 
``active'' points in the reference measure Gaussian kernel (typically $K=64$). 
Note that,  we extend the test cases to reference measures $\bP^h$ with the variable in time  $\bs_i = \sqrt{\bD_i}$. The finest computed 
discretisation in time corresponds to $1/h = N_t = 513$ but we will also provide 
convergence curves in $N_t$. 

The first test case, see figure  \ref{expe2},  captures a diffusion between 
Gaussian  source and target measures  in convex order.  The  means are  $\mu = 0.5$ and respective standard deviation $\sigma_0 = 0.05$ and $\sigma_1 = 0.1$ 
 (Figure  \ref{expe2} (a)). The  reference measure volatility  (b)  is constant and satisfies $0.009 =\bs = \sqrt{\bD} > \sqrt{\sigma_1^2 - \sigma_0^2} =0.0075$. 
 It is over-diffusive with respect to the target measure and  
not an admissible solution for the problem.  
Figure  \ref{expe2} (e) shows the method recovers the constant $\sigma = 0.0075$ solution.
The volatility drops at the boundary because we have truncated the domain and strongly penalized the drift  (f), the density is close to 0 there, see (c) for the curve in time of marginal densities. 
 Figure  \ref{expe2} (h) shows that the scaled relative entropy and discrete specific entropy get closer as $N_T$ the number of timesteps increase. 
 Figure  \ref{expe2} (g)  confirms that under our choice of truncation and space discretisation the cost of the algorithm is of order $O(N_T^{3/2})$. 
  
 The second  test case, see figure  \ref{expe4} is similar to  figure  \ref{expe2} but the target measure is a mixture of Gaussians (a).  
 $    \rho_1(x) = p \,\mathcal{N}(x; 0.5, 0.1) + \frac{1-p}{2} \, \mathcal{N}(x; 0.5 - d_1, 0.05) + \frac{1-p}{2} \mathcal{N}(x; 0.5 + d_1, 0.05) $
 ($p=0.6$).  The optimal volatility surface 
 The results are consistent with those of the first test case, we have a nice convergence in 
$N_T$ for the entropies (h).  The volatility surface (e) is not  simple. 

In the third test case, figure \ref{expe5}, the 
setting is similar to figure \ref{expe4} except for the reference volatility which is discontinuous in time  with a drop around time $0.5$:  $    \bs =         0.03$  if $ |t - 0.5| \le 0.2 $ and 
 $      0.01 $ else.  Again results are consistent with the previous test cases.  We can also observe the regularizing effect of the specific entropy on (e) (which is asymptotically close to a 
 weighted $\Lc^2$ norm). 

For the final test figure \ref{expe8}, the solution cannot be a martingale, the source and targets (a) are not in convex order.
$  \rho_0(x) = \mathcal{N}(x; 0.5, 0.1)  \mbox{ if } x \leq 0.5$ and $  \mathcal{N}(x; 0.5, 0.3) \mbox{ if } x > 0.5 $, 
$  \rho_1(x) = \mathcal{N}(x; 0.5, 0.5)  \mbox{ if } x \leq 0.5$ and $  \mathcal{N}(x; 0.5, 0.2) \mbox{ if } x > 0.5 $. 
The optimal solutions (c) (e) (f) generate diffusion where possible ($x< 0.5$) and use the drift in the non-convex order side ($x >0.5$).

\begin{figure}[ht]
    \centering
    \begin{subfigure}[t]{0.45\textwidth}
        \centering
        \includegraphics[width=\textwidth]{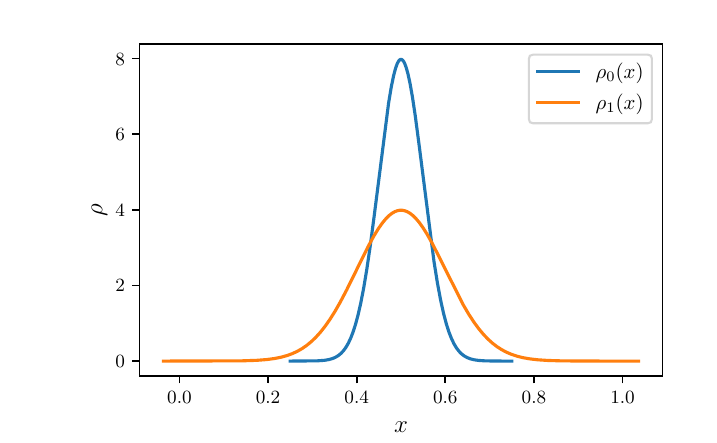}
        \caption{Source and target distributions}
        \label{fig:gauss_to_gauss_higher_source_target_distributions}
    \end{subfigure}
    \begin{subfigure}[t]{0.45\textwidth}
        \centering
        \includegraphics[width=\textwidth]{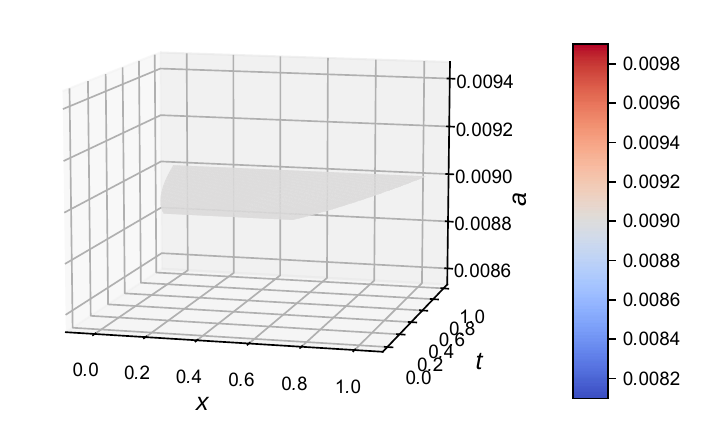}
        \caption{Reference measure volatility surface}
        \label{fig:gauss_to_gauss_higher_ref_mes_vol_surf}
    \end{subfigure}

    \begin{subfigure}[t]{0.45\textwidth}
        \centering
        \includegraphics[width=\textwidth]{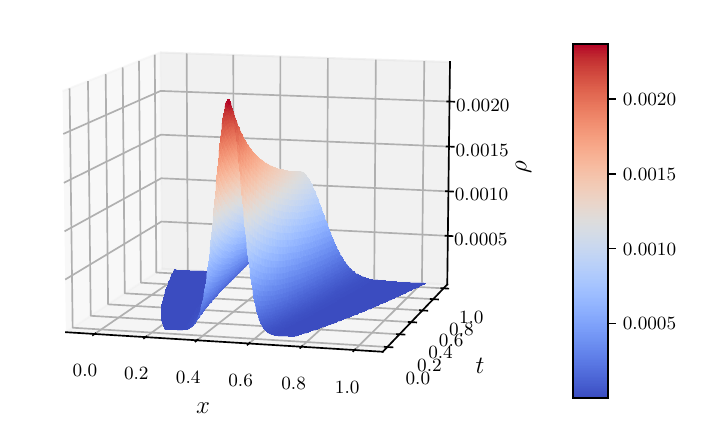}
        \caption{Marginal surface}
        \label{fig:gauss_to_gauss_higher_rel_marg_surf}
    \end{subfigure}
    \begin{subfigure}[t]{0.45\textwidth}
        \centering
        \includegraphics[width=\textwidth]{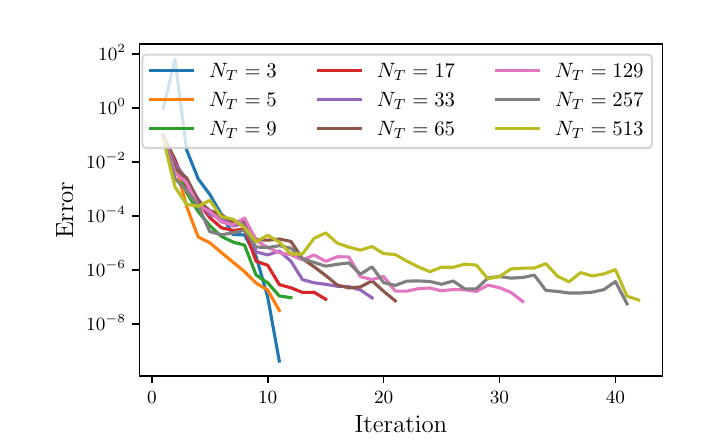}
        \caption{Sinkhorn convergence of the relative $L^\infty$ error}
        \label{fig:gauss_to_gauss_higher_rel_linf_conv}
    \end{subfigure}

    \begin{subfigure}[t]{0.45\textwidth}
        \centering
        \includegraphics[width=\textwidth]{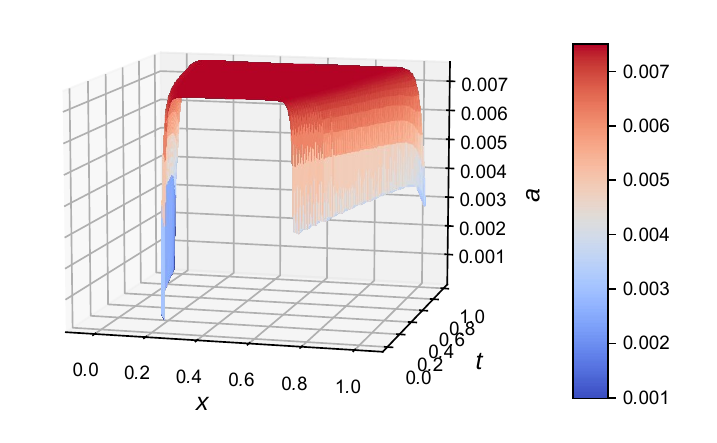}
        \caption{Final volatility surface}
        \label{fig:gauss_to_gauss_higher_vol_surf}
    \end{subfigure}
    \begin{subfigure}[t]{0.45\textwidth}
        \centering
        \includegraphics[width=\textwidth]{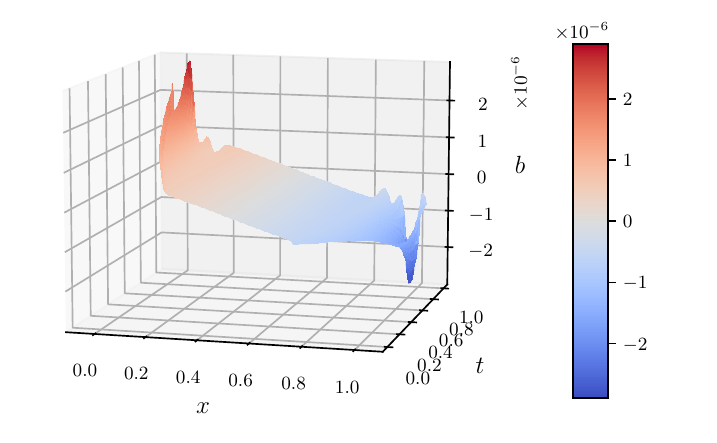}
        \caption{Final drift surface}
        \label{fig:gauss_to_gauss_higher_drift_surf}
    \end{subfigure}

    \begin{subfigure}[t]{0.45\textwidth}
        \centering
        \includegraphics[width=\textwidth]{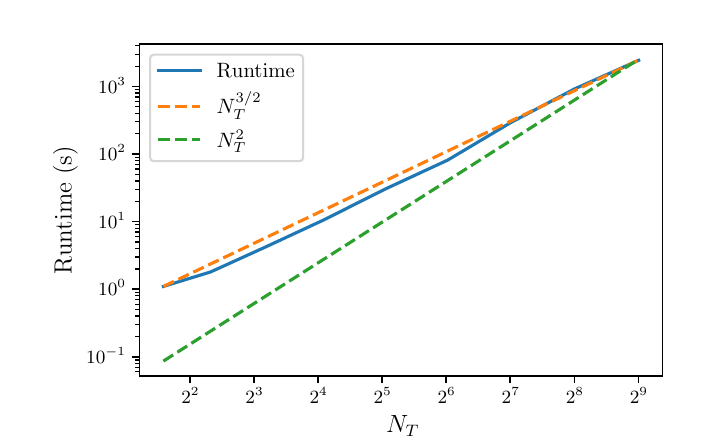}
        \caption{Runtime curve}
        \label{fig:gauss_to_gauss_higher_rt_curves}
    \end{subfigure}
    \begin{subfigure}[t]{0.45\textwidth}
        \centering
        \includegraphics[width=\textwidth]{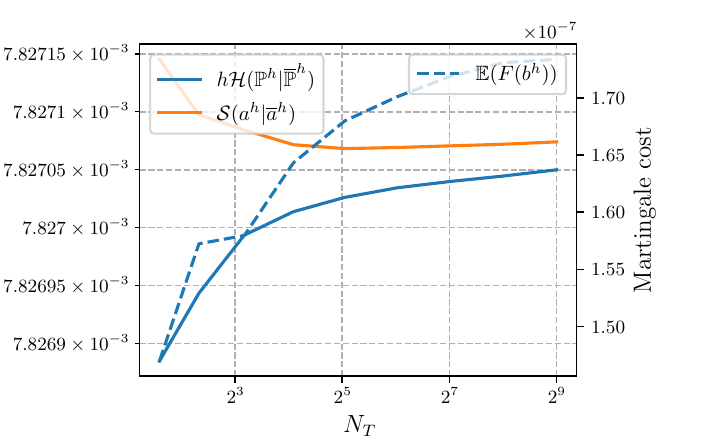}
        \caption{Convergence of the entropy and specific entropy}
        \label{fig:gauss_to_gauss_higher_conv_ent_specent}
    \end{subfigure}
    \caption{Gaussian to gaussian with overdiffusive  reference volatiliy}
    \label{expe2}

\end{figure}
 


\begin{figure}[ht]
    \centering
    \begin{subfigure}[t]{0.45\textwidth}
        \centering
        \includegraphics[width=\textwidth]{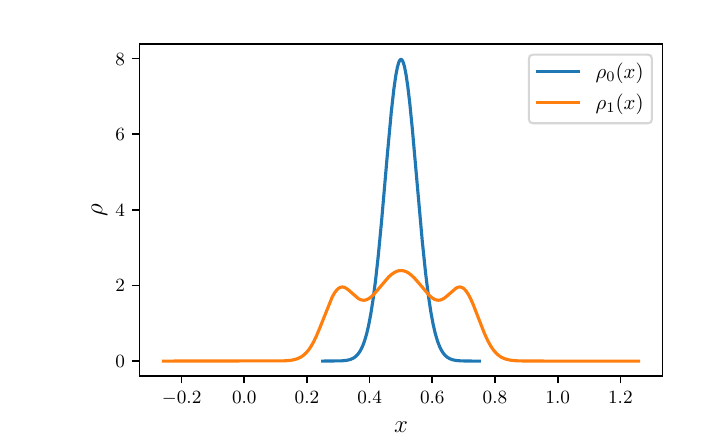}
        \caption{Source and target distributions}
        \label{fig:gauss_to_sumgauss_source_target_distributions}
    \end{subfigure}
    \begin{subfigure}[t]{0.45\textwidth}
        \centering
        \includegraphics[width=\textwidth]{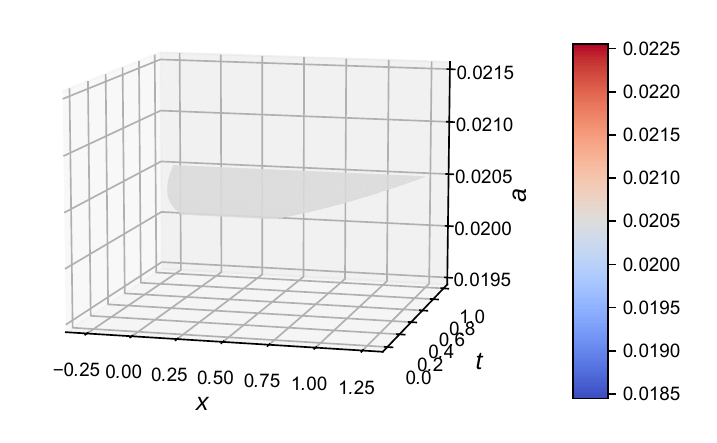}
        \caption{Reference measure volatility surface}
        \label{fig:gauss_to_sumgauss_ref_mes_vol_surf}
    \end{subfigure}

    \begin{subfigure}[t]{0.45\textwidth}
        \centering
        \includegraphics[width=\textwidth]{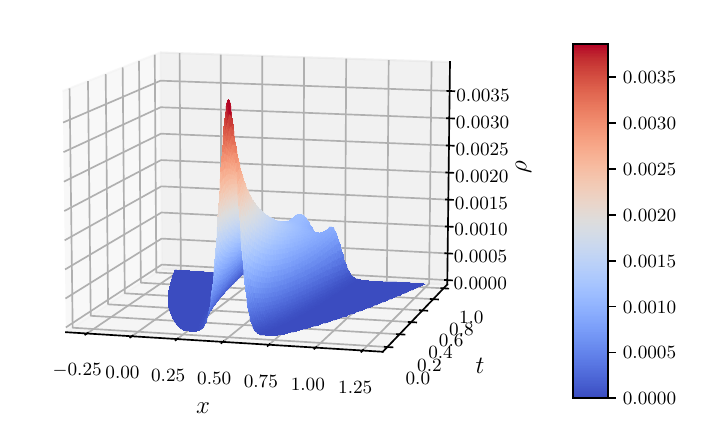}
        \caption{Marginal surface}
        \label{fig:gauss_to_sumgauss_rel_marg_surf}
    \end{subfigure}
    \begin{subfigure}[t]{0.45\textwidth}
        \centering
        \includegraphics[width=\textwidth]{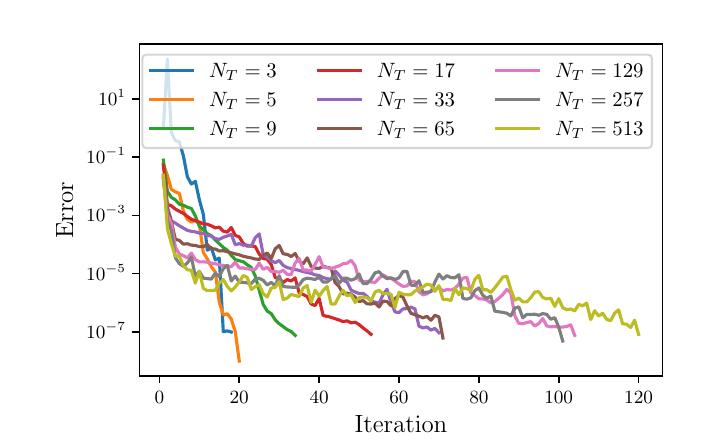}
        \caption{Sinkhorn onvergence of the relative $L^\infty$ error}
        \label{fig:gauss_to_sumgauss_rel_linf_conv}
    \end{subfigure}

    \begin{subfigure}[t]{0.45\textwidth}
        \centering
        \includegraphics[width=\textwidth]{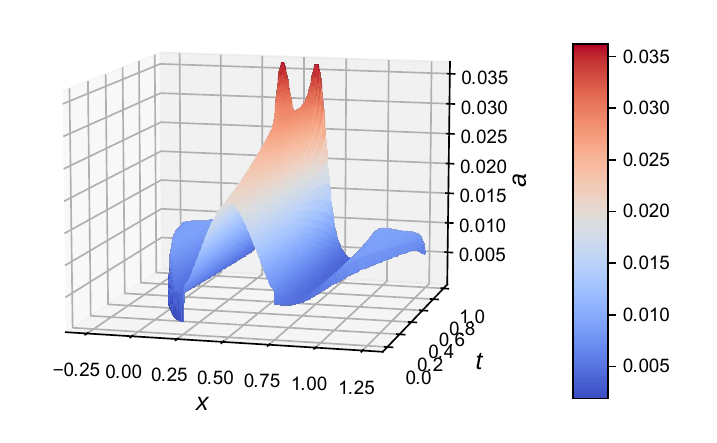}
        \caption{Final volatility surface}
        \label{fig:gauss_to_sumgauss_vol_surf}
    \end{subfigure}
    \begin{subfigure}[t]{0.45\textwidth}
        \centering
        \includegraphics[width=\textwidth]{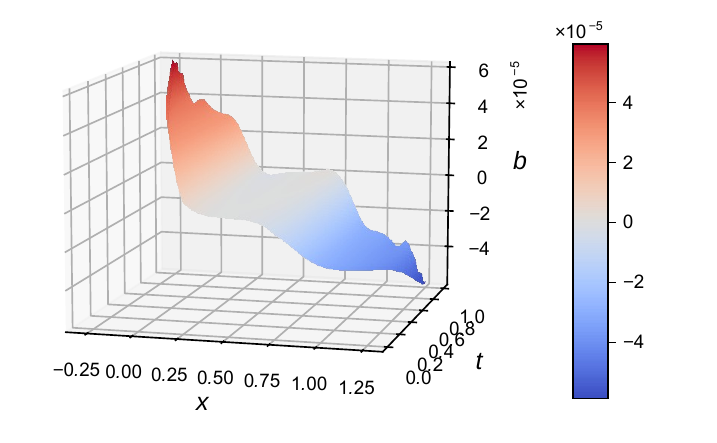}
        \caption{Final drift surface}
        \label{fig:gauss_to_sumgauss_drift_surf}
    \end{subfigure}

    \begin{subfigure}[t]{0.45\textwidth}
        \centering
        \includegraphics[width=\textwidth]{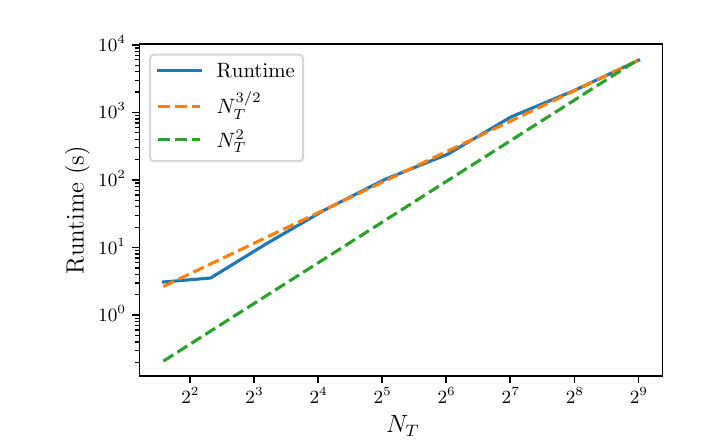}
        \caption{Runtime curve}
        \label{fig:gauss_to_sumgauss_rt_curves}
    \end{subfigure}
    \begin{subfigure}[t]{0.45\textwidth}
        \centering
        \includegraphics[width=\textwidth]{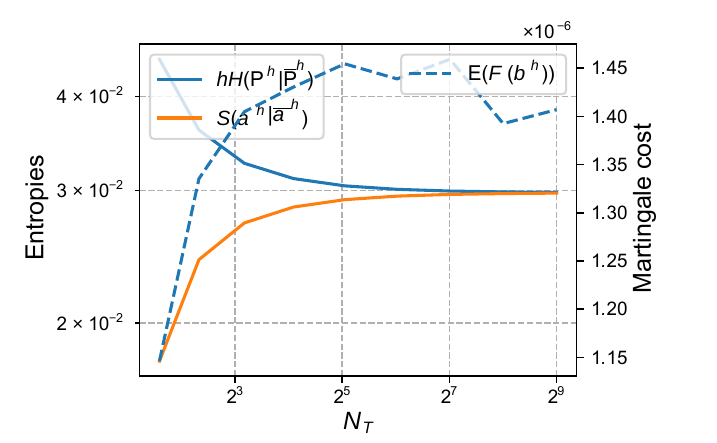}
        \caption{Convergence of the entropy and specific entropy}
        \label{fig:gauss_to_sumgauss_conv_ent_specent}
    \end{subfigure}
    \caption{A Gaussian to a mixture of  gaussian}
    \label{expe4} 
\end{figure} 
\begin{figure}[ht]
    \centering
    \begin{subfigure}[t]{0.45\textwidth}
        \centering
        \includegraphics[width=\textwidth]{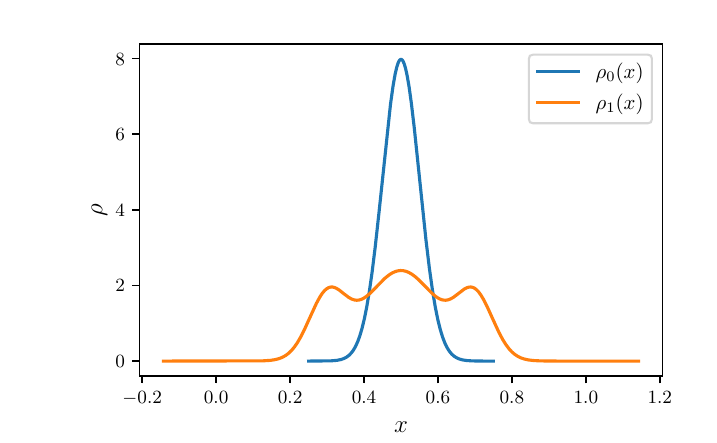}
        \caption{Source and target distributions}
        \label{fig:gauss_to_sumgauss_volup_source_target_distributions}
    \end{subfigure}
    \begin{subfigure}[t]{0.45\textwidth}
        \centering
        \includegraphics[width=\textwidth]{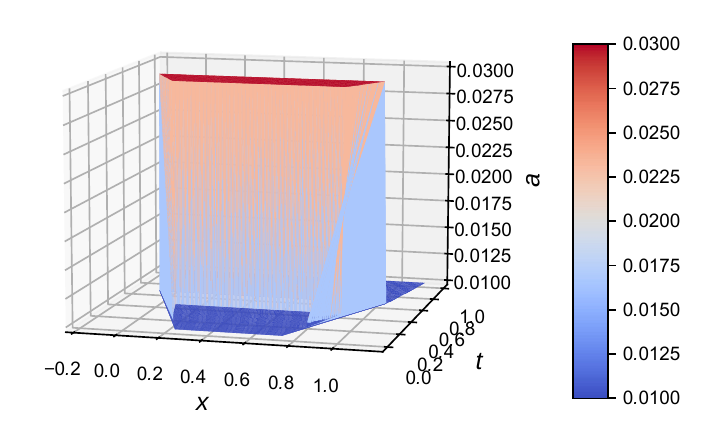}
        \caption{Reference measure volatility surface}
        \label{fig:gauss_to_sumgauss_volup_ref_mes_vol_surf}
    \end{subfigure}

    \begin{subfigure}[t]{0.45\textwidth}
        \centering
        \includegraphics[width=\textwidth]{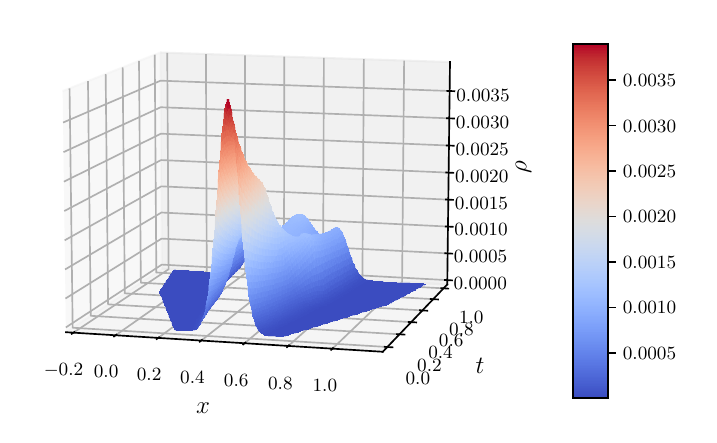}
        \caption{Marginal surface}
        \label{fig:gauss_to_sumgauss_volup_rel_marg_surf}
    \end{subfigure}
    \begin{subfigure}[t]{0.45\textwidth}
        \centering
        \includegraphics[width=\textwidth]{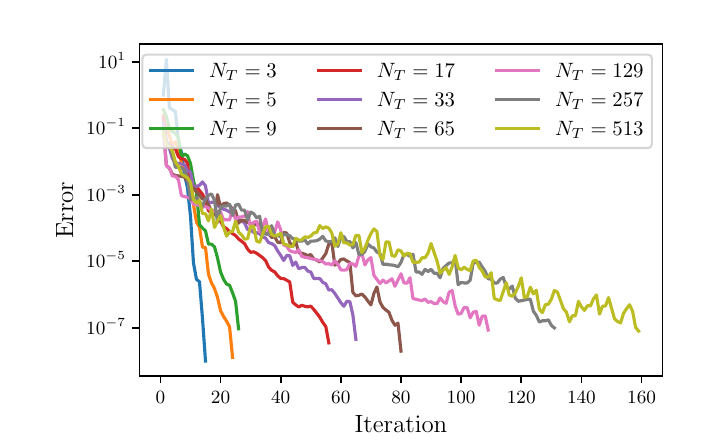}
        \caption{Sinkhorn onvergence of the relative $L^\infty$ error}
        \label{fig:gauss_to_sumgauss_volup_rel_linf_conv}
    \end{subfigure}

    \begin{subfigure}[t]{0.45\textwidth}
        \centering
        \includegraphics[width=\textwidth]{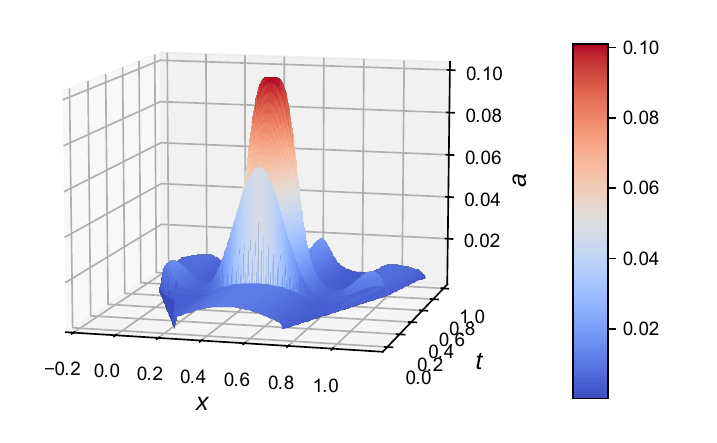}
        \caption{Final volatility surface}
        \label{fig:gauss_to_sumgauss_volup_vol_surf}
    \end{subfigure}
    \begin{subfigure}[t]{0.45\textwidth}
        \centering
        \includegraphics[width=\textwidth]{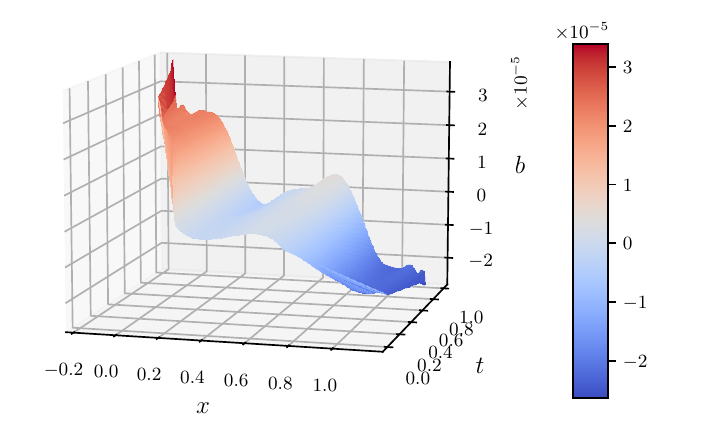}
        \caption{Final drift surface}
        \label{fig:gauss_to_sumgauss_volup_drift_surf}
    \end{subfigure}

    \begin{subfigure}[t]{0.45\textwidth}
        \centering
        \includegraphics[width=\textwidth]{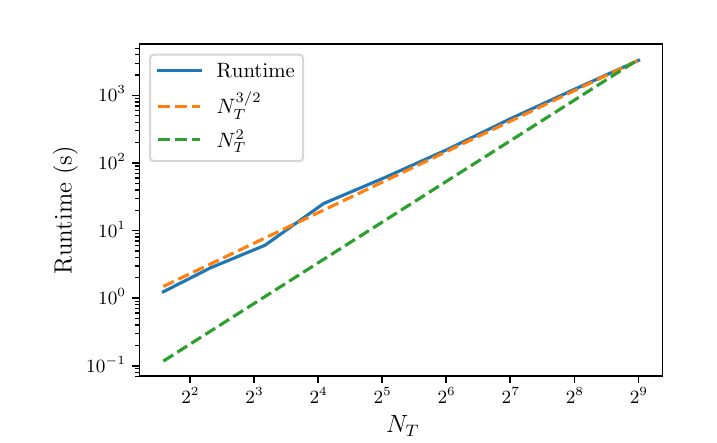}
        \caption{Runtime curve}
        \label{fig:gauss_to_sumgauss_volup_rt_curves}
    \end{subfigure}
    \begin{subfigure}[t]{0.45\textwidth}
        \centering
        \includegraphics[width=\textwidth]{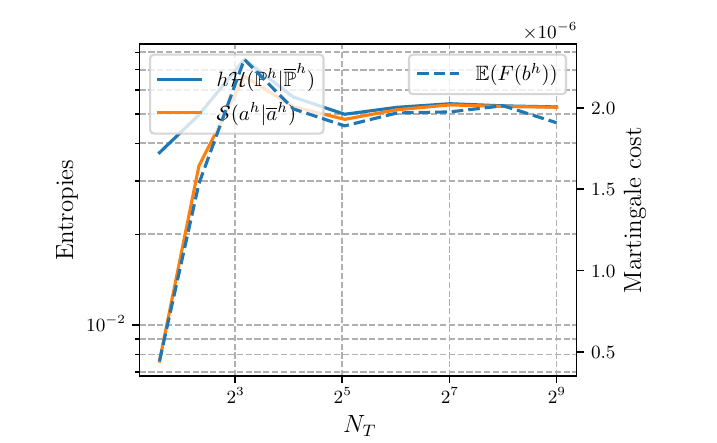}
        \caption{Convergence of the entropy and specific entropy}
        \label{fig:gauss_to_sumgauss_volup_conv_ent_specent}
    \end{subfigure}
    \caption{A Gaussian to sum of gaussians with a discontinuous in time volatility} 
\label{expe5}
\end{figure} 

\begin{figure}[ht]
    \centering
    \begin{subfigure}[t]{0.45\textwidth}
        \centering
        \includegraphics[width=\textwidth]{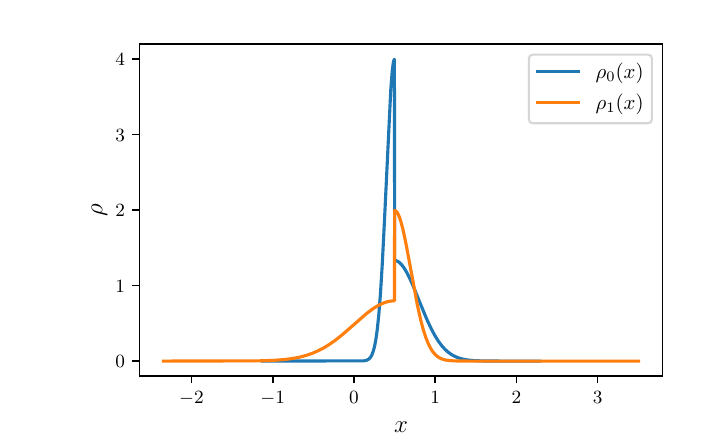}
        \caption{Source and target distributions}
        \label{fig:pw_gauss_1_source_target_distributions}
    \end{subfigure}
    \begin{subfigure}[t]{0.45\textwidth}
        \centering
        \includegraphics[width=\textwidth]{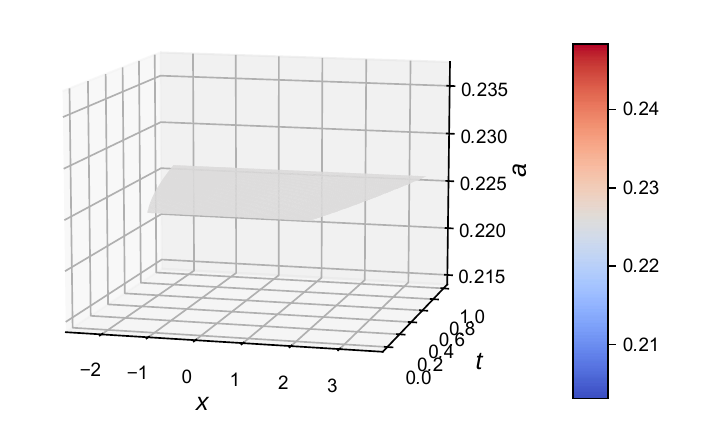}
        \caption{Reference measure volatility surface}
        \label{fig:pw_gauss_1_ref_mes_vol_surf}
    \end{subfigure}

    \begin{subfigure}[t]{0.45\textwidth}
        \centering
        \includegraphics[width=\textwidth]{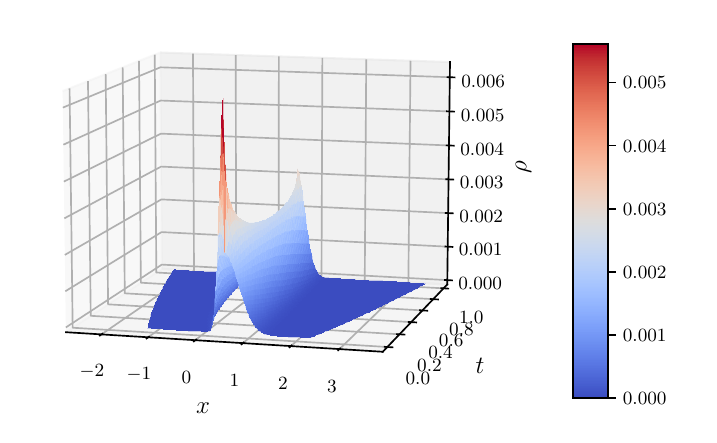}
        \caption{Marginal surface}
        \label{fig:pw_gauss_1_rel_marg_surf}
    \end{subfigure}
    \begin{subfigure}[t]{0.45\textwidth}
        \centering
        \includegraphics[width=\textwidth]{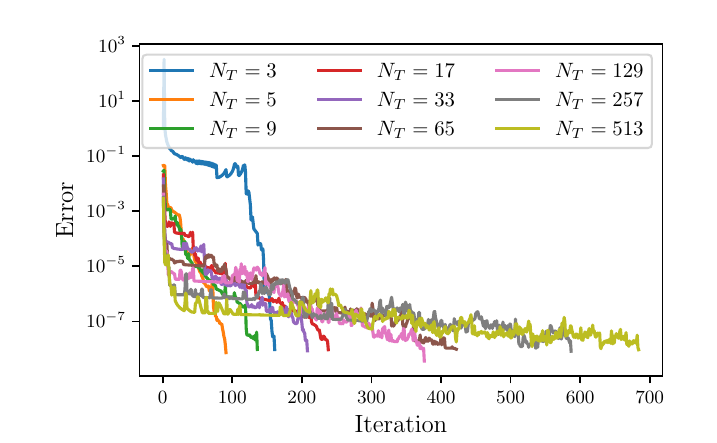}
        \caption{Convergence of the relative $L^\infty$ error}
        \label{fig:pw_gauss_1_rel_linf_conv}
    \end{subfigure}

    \begin{subfigure}[t]{0.45\textwidth}
        \centering
        \includegraphics[width=\textwidth]{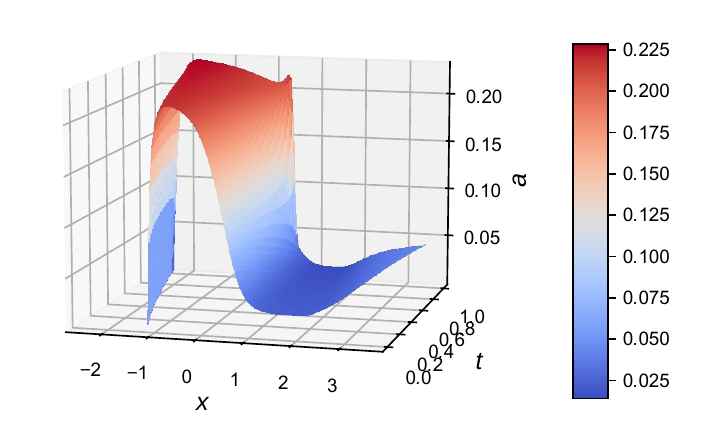}
        \caption{Final volatility surface}
        \label{fig:pw_gauss_1_vol_surf}
    \end{subfigure}
    \begin{subfigure}[t]{0.45\textwidth}
        \centering
        \includegraphics[width=\textwidth]{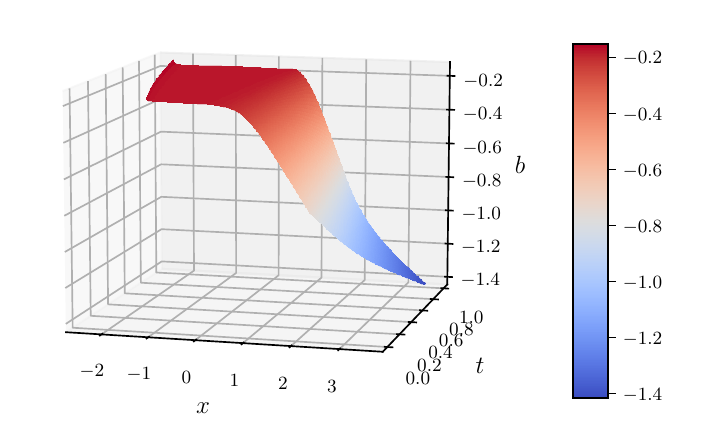}
        \caption{Final drift surface}
        \label{fig:pw_gauss_1_drift_surf}
    \end{subfigure}

    \begin{subfigure}[t]{0.45\textwidth}
        \centering
        \includegraphics[width=\textwidth]{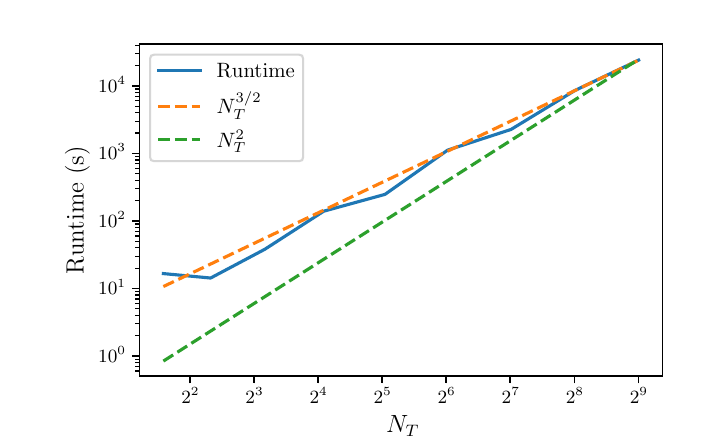}
        \caption{Runtime curve}
        \label{fig:pw_gauss_1_rt_curves}
    \end{subfigure}
    \begin{subfigure}[t]{0.45\textwidth}
        \centering
        \includegraphics[width=\textwidth]{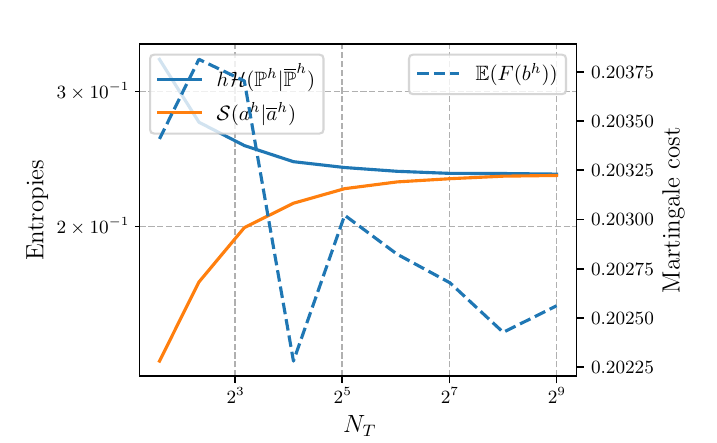}
        \caption{Convergence of the entropy and specific entropy}
        \label{fig:pw_gauss_1_conv_ent_specent}
    \end{subfigure}
    \caption{Marginals not in convex order}
    \label{expe8} 
\end{figure} 
 

\bibliography{bibben,newbib}%

\section{Appendix}
\label{annex} 

\subsection{Proof of proposition \ref{p1} [Specific Relative Entropy] } 
 \label{A1}


{\bf Proof of Prop. \ref{p1}  $i)$ } 
We have $X_{i}^h = X_{ih}$, where $(X_t)_{t \in [0,1]}$ solves 
$dX_t = b_t(X_t)dt+a_t(X_t)^{1/2}dB_t$.

\noindent {\it Step 1: We recall the ingredient we need for small-time asymptotics of the heat kernel}. Let $p_{a,b}^h(t,x,y)$ denote the transition density associated with the diffusion process
$dX_t = b(t,X_t)dt+\sigma(t,x)dB_t$
with values in $\mathbb R^d$ for some $d$-dimensional standard Brownian motion, and $a(t,x) = \sigma(t,x)\sigma(t,x)^\star$. We assume that both functions $(t,x) \mapsto b(t,x)$ and $(t,x) \mapsto a(t,x)$ are twice continuously differentiable, and that the matrix $a(t,x)$ is invertible for every $(t,x) \in [0,1] \times \mathbb R^d$. For $x,y \in \mathbb R^d$ and $0 \leq t \leq t+h \leq 1$, let 
$$p_{a,b}(t,h,x,y)dy= \mathbb P(X_{t+h}\in dy\,|\,X_t = x)$$
denote the transition probability associated with the diffusion process with drift $b$ and diffusion matrix $a$. From \cite{azencott1984densite}, the following asymptotic expansion is known: for every compact $U \subset \mathbb R^d$, there exists $\tau >0$ such that for $\|x-y\| \leq \tau$, we have
\begin{equation} \label{eq: basic transition expansion}
p_{a,b}(t,h,x,y) = \frac{1}{(2\pi h)^{d/2} \mathrm{det}(a(t,x)))^{1/2}}e^{\left(\tfrac{1}{2h}(y-x)^\star a(t,x)^{-1}(y-x)\right)}\big(1+hR(t,h,x,y)\big)
\end{equation}
where
$\sup_{0 \leq t \leq t+h\leq 1, x,y \in U}|R(t,h,x,y)| \leq C$
and $R$ only depends on $\partial_t^i\partial_x^j \sigma$ and $\partial_t^i\partial_x^j b$, for $0 \leq i+j \leq 2$, computed on $x$ and on $f_u$, $0 \leq u \leq 1$, where $f_u$ is the unit length geodesic joining $x$ and $y$ for the Riemannian metric associated with $a(t,x)^{-1}$.
\\

\noindent {\it Step 2: We separate the basic Kullback-Leibler divergence into a local and a global part}. Define
$q_a(t,h,x,y) = \frac{1}{(2\pi h)^{d/2} \mathrm{det}(a(t,x))^{1/2}}\exp{-\frac{1}{2h}(y-x)^\star a(t,x)^{-1}(y-x)}.$
From \eqref{eq: basic transition expansion}, a first-order expansion in $h$ yields
$
\log{ \frac{p_{a,b}(t,h,x,y)}{p_{\bar a,0}(t,h,x,y)} }  = \log{\frac{q_{a}(t,h,x,y)}{q_{\bar a}(t,h,x,y)} } +h {\widetilde R}(t,h,x,y)
$
where $\widetilde R$ has the same properties as $R$ with dependence upon $a$, $b$ and $\bar a$, {\it i.e.} it is valid for $\|x-y\| \leq \tau$. For $\|x-y\| \geq \tau$, we always have the Aaronson's type estimate \cite{aronson1967bounds} that reads
\begin{equation} \label{eq: aaronson}
c_-h^{-d/2}e^{-C_-h^{-1}\|x-y\|^2} \leq p_{a,b}(t,h,x,y) \leq c_+h^{-d/2}e^{-C_+h^{-1}\|x-y\|^2} 
\end{equation}
where $c_{\pm}, C_{\pm}$ depend on $\inf_{t,x} \|a(t,x)\| >0$ and $\sup_{t,x}(\|b(t,x)\|, \|a(t,x)\|)$, thus requiring uniform boundedness. 
Thus
\begin{align*}
\log{\frac{p_{a,b}(t,h,x,y)}{p_{\bar a,0}(t,h,x,y)}}  & =  \left(\log{\frac{q_{a}(t,h,x,y)}{q_{\bar a}(t,h,x,y)}}  +h {\widetilde R}(t,h,x,y)\right){\bf 1}_{\|x-y\| \leq \eta} + V_\eta(t,h,x,y),
\end{align*}
where
$V_\eta(t,h,x,y) = \log{\frac{p_{a,b}(t,h,x,y)}{p_{\bar a,0}(t,h,x,y)}} 
{\bf 1}_{\|x-y\| \geq \eta}.$
The estimate \eqref{eq: aaronson} yields
\begin{align*}
& {\mathbb E}_{\mathbb P}\left[V_\eta(ih,h,X_{ih},X_{(i+1)h})\right] \\
&\leq \frac{c^+_{a,b}}{c^-_{\bar a,0}}\mathbb P(\|X_{(i+1)h}-X_{ih}\| \geq \eta) -(C_{a,b}^+-C^-_{a,0})h^{-1} \mathbb E_{\mathbb P}[\|X_{(i+1)h}-X_{ih}\|^2{\bf 1}_{\|X_{(i+1)h}-X_{ih}\| \geq \eta}] \\
&\leq \big(\frac{c^+_{a,b}}{c^-_{\bar a,0}} -(C_{a,b}^+-C^-_{a,0})h^{-1} \mathbb E_{\mathbb P}[\|X_{(i+1)h}-X_{ih}\|^4]^{1/2}\big)\eta^{-1/2}\mathbb E\big[\|X_{(i+1)h}-X_{ih}\|]^{1/2} =O(h^{1/4}).
\end{align*}
Likewise,
$ \mathbb E_{\mathbb P}\big[h \widetilde R(ih,h,X_{ih},X_{(i+1)h})\big){\bf 1}_{\|X_{(i+1)h}-X_{ih}\| \leq \eta}\big] \leq Ch,
$
so that
\begin{equation} \label{eq: approx entropy}
\mathbb E_{\mathbb P}\big[\log{ \frac{p_{a,b}(ih,h,X_{ih},X_{(i+1)h})}{p_{\bar a,0}(ih,h,X_{ih},X_{(i+1)h})}\big] } ] = \mathbb E_{\mathbb P}\big[\log{  \frac{q_{a,b}(ih,h,X_{ih},X_{(i+1)h})}{q_{\bar a,0}(ih,h,X_{(i+1)h},X_{ih})} } \big] +O(h^{1/4})
\end{equation}
and therefore, according to \eqref{eq: approx entropy},
\begin{align*}
\mathcal H(\mathbb P^h\,|\,\overline{\mathbb P}^h) & =  \mathbb E_{\mathbb P}\Big[\log{ \prod_{i = 1}^{N}\frac{p_{a,b}(ih,h,X_{ih},X_{(i+1)h})}{p_{\bar a,0}(ih,h,X_{ih},X_{(i+1)h})} }\Big] \\
& = \sum_{i =1}^N  \mathbb E_{\mathbb P}\Big[\log { \frac{q_{a}(ih,h,X_{ih},X_{(i+1)h})}{q_{\bar a}(ih,h,X_{ih},X_{(i+1)h})} } \Big] + O(Nh^{1/4})\,.
\end{align*}

\noindent {\it Step 3. Control of the entropy of the proxy transitions.}
We have
\begin{align*}
&\log{ \frac{q_{a}(ih,h,X_{ih},X_{(i+1)h})}{q_{\bar a}(ih,h,X_{ih},X_{(i+1)h})} } 
= -\tfrac{1}{2}\log{  \frac{\mathrm{det}(a(ih,X_{ih}))}{\mathrm{det}(\bar a(ih, X_{ih}))}} \\
&-\tfrac{1}{2}h^{-1}(X_{(i+1)h}-X_{ih})^\star (a(ih,X_{ih})^{-1}-\bar a(ih,X_{ih})^{-1})(X_{(i+1)h}-X_{ih}).
\end{align*}
Next, we have
\begin{align*}
&\tfrac{1}{2}h^{-1}(X_{(i+1)h}-X_{ih})^\star (a(ih,X_{ih})^{-1}-\bar a(ih,X_{ih})^{-1})(X_{(i+1)h}-X_{ih}) \\
& = -\tfrac{1}{2}h^{-1}\big(\int_{ih}^{(i+1)h}\sigma(s,X_s)dB_s\big)^\star (a(ih,X_{ih})^{-1}-\bar a(ih,X_{ih})^{-1})\big(\int_{ih}^{(i+1)h}\sigma(s,X_s)dB_s\big) +\zeta_i^h,
\end{align*}
where
\begin{align*}
\zeta_i^h & = h^{-1}\big(\int_{ih}^{(i+1)h}b(s,X_s)ds\big)^\star (a(ih,X_{ih})^{-1}-\bar a(ih,X_{ih})^{-1})\big(\int_{ih}^{(i+1)h}b(s,X_s)ds\big) \\
+ & h^{-1}\big(\int_{ih}^{(i+1)h}\sigma(s,X_s)dB_s\big)^\star (a(ih,X_{ih})^{-1}-\bar a(ih,X_{ih})^{-1})\big(\int_{ih}^{(i+1)h}b(s,X_s)ds\big) \\
+ & h^{-1}\big(\int_{ih}^{(i+1)h}b(s,X_s)ds\big)^\star (a(ih,X_{ih})^{-1}-\bar a(ih,X_{ih})^{-1})\big(\int_{ih}^{(i+1)h}\sigma(s,X_s)dB_s\big)
\end{align*}
The boundedness of $b$ and $\sigma$ (and its inverse) readily yields 
$\max_{1 \leq i \leq N}\mathbb E_{\mathbb P}[\|\zeta_i^h\|] = O(h^{1/2})$
by applying Cauchy-Schwarz inequality repeatedly. Also, 
writing
$h^{-1/2}\int_{ih}^{(i+1)h}\sigma(s,X_s)dB_s = \sigma(ih,X_{ih})+\xi_i^h,$
where 
$\xi_i^h = h^{-1/2}\int_{ih}^{(i+1)h}(\sigma(s,X_s)-\sigma(ih,X_{ih}))dB_s$
satisfies
$\max_{1 \leq i \leq N}\mathbb E_{\mathbb P}[\|\xi_i^h\|] = O(h^{1/2})$
thanks to the smoothness of $(t,x) \mapsto \sigma(t,x)$ (at least Lipschitz in both variables) and standard moment estimates, we obtain
\begin{align*}
&-\tfrac{1}{2}h^{-1}\big(\int_{ih}^{(i+1)h}\sigma(s,X_s)dB_s\big)^\star (a(ih,X_{ih})^{-1}-\bar a(ih,X_{ih})^{-1})\big(\int_{ih}^{(i+1)h}\sigma(s,X_s)dB_s\big) \\
& = \tfrac{1}{2} \mathrm{Tr}\big(\bar a(ih,X_{ih})^{-1}(a(ih,X_{ih})-\bar a(ih,X_{ih)}))\big) + \widetilde \xi_i^h,
\end{align*}
where now, in the same way as for the remainder term $\xi_i^h$, we readily have
$\max_{1 \leq i \leq N}\mathbb E_{\mathbb P}[|\widetilde \xi_i^h|] = O(h^{1/2}).$
Putting together all our estimates, we obtain
\begin{align*}
&\log{ \frac{q_{a}(ih,h,X_{ih},X_{(i+1)h})}{q_{\bar a}(ih,h,X_{ih},X_{(i+1)h})} } 
= -\tfrac{1}{2}\log{  \frac{\mathrm{det}(a(ih,X_{ih}))}{\mathrm{det}(\bar a(ih, X_{ih}))}} \\
&+\tfrac{1}{2} \mathrm{Tr}\big(\bar a(ih,X_{ih})^{-1}(a(ih,X_{ih})-\bar a(ih,X_{ih)}))\big) +\widetilde \zeta_i^h,
\end{align*}
for some stochastic remainder term $\widetilde \zeta_i^h$ that satisfies the moment estimate
$\max_{1 \leq i \leq N}\mathbb E_{\mathbb P}[|\widetilde \zeta_i^h|] = O(h^{1/2}).$
By definition of the relative entropy $\mathcal S^{\mathcal I}$, we thus obtain
\begin{align*}
\mathcal H(\mathbb P^h\,|\,\overline{\mathbb P}^h) = \frac{1}{2}\sum_{i = 1}^N \mathbb E_{\mathbb P}\big[\mathcal S^{\mathcal I}(a(ih,X_{ih})|\bar a(ih,X_{ih})\big]+O(Nh^{1/4}).
\end{align*}
Finally, using that $h \sim N^{-1}$, we obtain
\begin{equation}
\label{MarcA}
h \, \mathcal H(\mathbb P^h\,|\,\overline{\mathbb P}^h) = \frac{1}{2}\sum_{i = 1}^N h  \mathbb E_{\mathbb P}\big[\mathcal S^{\mathcal I}(a(ih,X_{ih})|\bar a(ih,X_{ih})\big]+O(h^{1/4})
\end{equation}
and 
$\lim_{h \rightarrow 0}h\mathcal H(\mathbb P^h\,|\,\overline{\mathbb P}^h) =\mathbb E_{\mathbb P}\Big(  \int_0^1 \mathcal S^{\mathcal I}(a(s,X_{s})|\bar a(s,X_{s})ds\Big)$
by Riemann approximation, which is the desired result. More specifically, one readily checks that
$$\int_{ih}^{(i+1)h}\mathcal S^{\mathcal I}(a(s,X_{s})|\bar a(s,X_{s})ds = h\mathcal S^{\mathcal I}(a(ih,X_{ih})|\bar a(ih,X_{ih})+\rho_i^h,$$ 
where $\max_{1 \leq i \leq N}\mathbb E[|\rho_i^h|] = O(h^{1/2})$
using that both $\mathcal S^{\mathcal I}$ and $a$, $\bar a$ are smooth functions.\\

\noindent {\bf Proof of Proposition \ref{p1} $ii)$ } 
Since $\P^h$ is Markov we rewrite the relative Entropy using the transition probabilities 
$
  \Ec(\P^h | \bP^h)   =  \Ec(\P_0 | \rho_0) +  \sum_{i=0}^{N-1} \E_{\P^h_i}(  \Ec(\P_\tri^h | \bP_\tri^h) )  \,.
$  
It is therefore sufficient to focus on one transition and show 
$
\Ec(\P_\tri^h | \bP_\tri^h) \ge \dfrac{1}{2} \S^{\Ic}(\D_i^h(X_i^h) | \bD) \,. 
$
Let us simplify the notations and look  at $\Ec(\P_x | \bP_x)$ where $\P_x$ is  a probability measure over a set $Y=\R_{t_{i+1}}$
parameterized by $x \in X=\R_{t_{i}}$ and $\bP_x$ a probability measure over $Y$ with normal  law $\Gamma_{x\rightarrow y} (0,\bD)$.  
This inequality is a consequence of the well-known fact that the probability that minimizes the entropy under constrained second moments is a Gaussian measure. In the following, we prove it using the dual formulation of entropy, recalled hereafter, see \cite[lemma 9.4.4.]{AGS}
\begin{equation} \label{dualEntropy}
\left\{  
\begin{array}{c}
  \Ec(\P_x | \bP_x)  = \sup_{\displaystyle f \in C_b^0(Y)}  \beta(f) \\
  \beta(f) := \dint f(y)\, d\P_x(y)  - \dint \exp{f(y)} -1  \, d\bP_x(y)
  \,.\end{array}
  \right.
\end{equation}
We plan on restricting the test functions $f$ to quadratic functions 
$ y \rightarrow p_{x,C,A} (y) = C+ A\, (y-x)^2
$
which are not bounded functions. To do so, we remark that the truncation 
\begin{equation}
        p_{x,C,A,R} (y) = \begin{cases} p_{x,C,A} (y) &\text{ if } |y-x| \leq R\\
         C+ A\, R^2 &\text{ if } | y- x | \geq R\,,
    \end{cases}
\end{equation}
is a pointwise increasing (in the variable $R$) sequence of functions if $A \geq 0$ and it is also the case for $\exp{p_{x,C,A,R}}$. 
It is a pointwise decreasing sequence of functions if $A \leq 0$ and it is also the case for $\exp{p_{x,C,A,R}}$. As a consequence, in both cases, one can apply the monotone convergence theorem to obtain 
\begin{equation}\label{EqMonotoneTheorem}
\beta(p_{x,C,A}) =  \lim_{R \to \infty} \beta(p_{x,C,A,R}) \leq  \Ec(\P_x | \bP_x) \,.
\end{equation}
We can now directly optimize on the family of quadratic functions.
plugging 
$f:= C+ A\,(y-x)^2 $ in the integrals we get 
\[
\begin{array}{ll}
\beta(C+ A\, \dfrac{(y-x)^2}{2} ) &  = C + h\, \dfrac{D(x)}{2}  \, A  -  \dint \exp{C+A\,\dfrac{(y-x)^2}{2}} -1  \, d\bP_{x}(y)  \\
\end{array}
\]
$(C,A) \rightarrow \beta(C+ A\,(y-x)^2)$ is strictly concave. The optimality in $C$ gives 
\begin{equation}
  1   =  (2 \pi h \bD )^{-\frac{1}{2}} \, \exp{C} \, \dint \exp{(A - \dfrac{1}{h\, \bD}) \,\dfrac{(y-x)^2}{2}  } \, dy \,.\\[8pt]   
  \end{equation}
We identify above  a Gaussian probability measure over $y$ with mean $x$ and standard deviation $\alpha$ defined by 
$
\alpha^2(A)   = (\dfrac{1}{h\, \bD} -A )^{-1}
$
(for small $h$ and $\bD > 0$ this is always well defined). 
The  normalizing constant gives:
$
(2 \pi h \bD )^{-\frac{1}{2}} \, \exp{C}  = (2 \pi \alpha^2(A)   )^{-\frac{1}{2}} 
$.
We can eliminate $C$ and the function to maximize in $A$ now is (the second integral vanishes) 
$
A \rightarrow  \dfrac{1}{2}  \log{1-h\, \bD \, A } + h\, \dfrac{D(x)}{2}  \, A  
$
the optimal $A$ satisfies
$
1-h\, \bD \, A  = \dfrac{\bD}{D} \quad \quad h\, D \, A = \dfrac{D(x)}{\bD} - 1 
\,.$
Therefore we get the result using the inequality \eqref{EqMonotoneTheorem}
\[
\sup_{ \left\{ \begin{array}{l} 
(C,A) \in (\R \times \R) \\
f:= C+ A\,(y-x)^2
\end{array} \right.
}  \beta(f)  =  - \log{ \dfrac{D(x)}{\bD}  } + \dfrac{D(x)}{\bD} - 1 \,.
\]

\subsection{Proof of  theorem \ref{ThStroockVaradhanExtendedInTime} [Convergence of Markov Chain to diffusion] }
 \label{A2}

\subsection{Convergence of Markov chains to diffusions} 


We consider an inhomogeneous (family in $h$) of transition probabilities $\P_{\tri}^h(x,dy)$ on $\R^d$ for $i=0,\ldots, N-1$. We associate to it a discrete time Markov chain $X_0, \ldots, X_{N-1}$ that is turned itself into a continuous time process $x(t)$ with $x(ih)=X_i$ and linear interpolation for $t \in [ih, (i+1)h)$ for $i=0,\ldots, N-1$ with law $ \tP^h \in \Pc(\Omega)$.
 The discrete sampling $x(0), x(h), \ldots, x((N-1)h)$ is a Markov sequence with inhomogeneous transition 
$$\tP^h(x((i+1)h) \in dy\,|\,x(ih)=x) = \P_{\tri}^h(x,dy)\,.$$  

 We define the piecewise constant in time drift and diffusion coefficients~:

\begin{equation} \label{cSMOM}
\left\{
 \begin{array}{l} 
  \tV^h_t(x)  \cqq   \displaystyle \tfrac{1}{h}\mathbb E_{\mathbb P_{\lfloor t\,h^{-1}\rfloor \rightarrow \lfloor t\, h^{-1}\rfloor+1}}^h  \big[(X_{i+1}-x)1_{\|X_{i+1}-x\| \leq 1} \big],  \\[10pt]
 \tD^h_t(x)    \cqq    \displaystyle \tfrac{1}{h}\mathbb E_{\mathbb P_{\lfloor t\,h^{-1}\rfloor \rightarrow \lfloor t\, h^{-1}\rfloor+1}}^h  \big[(X_{i+1}-x)(X_{i+1}-x)^\star 1_{\|X_{i+1}-x\| \leq 1}\big]\,. \\ \\
  \end{array} 
  \right.
 \end{equation}  

We also need a conditional moment of order larger than $4$ ($\alpha >0$) referred to  abusively 
as ``kurtosis'' (kurtosis corresponds to $\alpha =0$) 
 \begin{equation} \label{KURT}
  \begin{array}{l}
  \K^h_i(x_i)  \coloneqq   h^{-(2+\alpha)}\mathbb E_{\P^h_\tri} ( \| X_{i+1}-x_i\|^{ 4+ 2\,\alpha} )\,. 
  \end{array}  
\end{equation}

{Note that when  $X_i^h$ is the discrete sampling of a continuous diffusion with bounded characteristics, $c_i^h(x_i)$ is controlled uniformly by the bound on the characteristics.
(see Lemma \ref{lemmaMH})}.


\begin{theorem}{}\label{ThStroockVaradhanExtendedInTime}
Assuming  $(\tV^h_t, \tD^h_t)$ defined in (\ref{cSMOM}) and $(\K_i^h)$ defined in  (\ref{KURT}) satisfy 
\begin{equation} \label{DDD11} 
\left\{
\begin{array}{l}
i)\quad  \lim_{h  \searrow 0}  \sup_{ |x|\le R, t \in [0,1]} \| \tV^h_t(x) - \tV^0_t(x)   \|  = 0 \,,   \\[10pt]
ii)\quad  \lim_{h  \searrow 0}  \sup_{ |x|\le R, t \in [0,1]} \| \tD^h_t(x) - \tD^0_t(x)   \|  = 0  \,,\\[10pt]
iii)\quad \lim_{h \searrow 0 } \max_{i=1,\ldots, N-1}\sup_{\|x\| \leq R}h\, \K^h_i(x) =  0  \,,\\[10pt]
\end{array}
\right.
\end{equation}
for some pair $\tV^0_t(x), \tD^0_t(x) $ of integrable functions.
Then,   $\tP^h$ narrowly converges to $\tP^0$ a diffusion process in the weak sense 
with drift and diffusion coefficients $(\tV^0_t, \tD^0_t)$ .
\end{theorem}
This is a generalisation of \cite[Th. 11.2.3]{SVbook} to inhomogeneous in time Markov chain. 
We give here some remarks about  space-time transformations in the inhomogeneous case that are not detailed in \cite{ SVbook} .

Suppose we have a inhomogeneous (family in $h$) of transition probabilities $\P_{\tri}^h(x,dy)$ on $\R^d$ for $i=0,\ldots, N-1$. We associate to it a discrete-time Markov chain $X_0, \ldots, X_{N-1}$ that is turned itself into a continuous time process $x(t)$ with $x(ih)=X_i$ and linear interpolation for $t \in [ih, (i+1)h)$ for $i=0,\ldots, N-1$. The discrete sampling $x(0), x(h), \ldots, x((N-1)h))$ is a Markov sequence with inhomogeneous transition 
$$\P(x((i+1)h) \in dy\,|\,x(ih)=x) = \P_{\tri}^h(x,dy).$$
We next consider an equivalent homogeneous Markov process/chain via a simple time/space transformation. 
Define the deterministic process over the integers $N(i) = i$ extended into a continuous time (deterministic) process 
$n(ih)=N(i)h=ih$ over the times of the form $ih$ and with (trivial) linear interpolation so that simply $n(t)=t$. Define the new process
$y(t) = (n(t), x(t)).$
We claim that $y_t$ is a homogeneous Markov sequence on the times of the form $ih$ with state space $[0,1]\times \R^d$ (actually $\{kh,k=0,\ldots, N-1\} \times \R^d$ but we embed everything into $[0,1]\times \R^d$). Indeed
\begin{multline*}
\P(y((i+1)h) \in dv dy\,|\,y(ih) = (u,x) ) \\
 = \P(n((i+1)h)\in dv, x((i+1)h) \in dy\,|\,n(ih) = u, x(ih) = x) )
\\ = \delta_{u+h}(dv)\P_{\lfloor uh^{-1}\rfloor \rightarrow \lfloor uh^{-1}\rfloor+1}^h(x, dy) 
 =: \Pi^h((u,x), dvdy),
\end{multline*}
and the last expression does not depend on $i$ hence the homogeneity.
With the same notation as Stroock-Varadhan, this defines the family of (homogeneous) transitions over which we are going to apply the standard homogeneous Stroock-Varadhan result. It suffices to show that our condition \eqref{DDD11} entails the three classical conditions of Stroock-Varadhan, i.e. conditions (2.4)-(2.5)-(2.6) p. 268.

Following the definition given at the end of p. 267 in Stroock-Varadhan, we have semi-explicit expressions for the drift and diffusion coefficients, (with $x \in \R^d$ that becomes $(u,x) \in [0,1] \times \R^d$) namely
\begin{align*}
b_h((u,x)) & = \tfrac{1}{h} \int_{|(v,y)-(u,x)| \leq 1} ((v,y)-(u,x))\Pi^h((u,x), dvdy) \\
& = \tfrac{1}{h}\left\{
\begin{array}{l}
\displaystyle\int_{|y-x| \leq 1} \int_{[0,1]} (v-u)\delta_{u+h}(dv)\P_{\lfloor uh^{-1}\rfloor \rightarrow \lfloor uh^{-1}\rfloor+1}^h(x, dy)  \\ \\
\displaystyle\int_{|y-x| \leq 1} (y-x) \int_{[0,1]}\delta_{u+h}(dv)\P_{\lfloor uh^{-1}\rfloor \rightarrow \lfloor uh^{-1}\rfloor+1}^h(x, dy)
\end{array}
\right. \\
& = \left\{
\begin{array}{l}
\displaystyle \P_{\lfloor uh^{-1}\rfloor \rightarrow \lfloor uh^{-1}\rfloor+1}^h(x, \mathcal B_{\R^d}(x,1))  \\ \\
\displaystyle \tfrac{1}{h}\int_{|y-x| \leq 1} (y-x)\P_{\lfloor uh^{-1}\rfloor \rightarrow \lfloor uh^{-1}\rfloor+1}^h(x, dy)
\end{array}
\right.
\end{align*}
and conditions $i)$ and $iv)$ of \eqref{DDD11} entail the convergence, for every $R>0$:
$\lim_{h \rightarrow 0}\sup_{u \in [0,1], |x| \leq R}\|b_h(u,x)-(p_u(x),\tilde b^0_u(x))\|=0$
which is (2.4) of Stroock-Varadhan with the state variable $(u,x) \in [0,1 \times \R^d$ instead of $x \in \R^d$.
Likewise, with a slight abuse of notation for the three components of the $2 \times 2$ symmetric matrix that defines the diffusion matrix, we have 
\begin{align*}
a_h((u,x)) 
& = \tfrac{1}{h}\left\{
\begin{array}{l}
\displaystyle\int_{|y-x| \leq 1} \int_{[0,1]} (v-u)^2\delta_{u+h}(dv)\P_{\lfloor uh^{-1}\rfloor \rightarrow \lfloor uh^{-1}\rfloor+1}^h(x, dy)  \\ \\
\displaystyle\int_{|y-x| \leq 1} (y-x)^2 \int_{[0,1]}\delta_{u+h}(dv)\P_{\lfloor uh^{-1}\rfloor \rightarrow \lfloor uh^{-1}\rfloor+1}^h(x, dy) \\ \\
\displaystyle\int_{|y-x| \leq 1} (y-x) \int_{[0,1]}(v-u)\delta_{u+h}(dv)\P_{\lfloor uh^{-1}\rfloor \rightarrow \lfloor uh^{-1}\rfloor+1}^h(x, dy)
\end{array}
\right. \\
& = \left\{
\begin{array}{l}
\displaystyle h\P_{\lfloor uh^{-1}\rfloor \rightarrow \lfloor uh^{-1}\rfloor+1}^h(x, \mathcal B_{\R^d}(x,1))  \\ \\
\displaystyle \tfrac{1}{h}\int_{|y-x| \leq 1} (y-x)^2\P_{\lfloor uh^{-1}\rfloor \rightarrow \lfloor uh^{-1}\rfloor+1}^h(x, dy) 
\\ \\
\displaystyle\int_{|y-x| \leq 1} (y-x)\P_{\lfloor uh^{-1}\rfloor \rightarrow \lfloor uh^{-1}\rfloor+1}^h(x, dy)
\end{array}
\right. \\ \\
& \rightarrow_{h \rightarrow 0} \left\{
\begin{array}{l}
0  \\ \\
\displaystyle \tfrac{1}{h}\int_{|y-x| \leq 1} (y-x)^2\P_{\lfloor uh^{-1}\rfloor \rightarrow \lfloor uh^{-1}\rfloor+1}^h(x, dy) 
\\ \\
0
\end{array}
\right.
\end{align*}
and conditions $ii)$ of \eqref{DDD11} entail the convergence
$\lim_{h \rightarrow 0}\sup_{u \in [0,1], |x| \leq R)}\|a_h(u,x)-
\left(
\begin{array}{ll}
\tilde a^0(x) & 0\\
0 & 0
\end{array}
\right)
\|=0$
which is (2.5) of Stroock-Varadhan. Finally, we check (2.6) of Stroock-Varadhan, namely
for $\ee>0$ given 
\begin{equation}
\label{SV33}
    \lim_{h \rightarrow 0}\sup_{|x| \leq R, u \in [0,1]}\Delta_h^\varepsilon((u,x), [0,1]\times\R^d \setminus \mathcal B_{[0,T]\times \R^d}((u,x), \varepsilon))=0. 
\end{equation}
Setting $\mathcal B^c(u,x,\varepsilon) = [0,1]\times\R^d \setminus \mathcal B_{[0,T]\times \R^d}((u,x), \varepsilon)$, we have
\begin{equation}
 \nonumber
\begin{array}{ll}
 \Delta_h^\varepsilon((u,x),\mathcal B^c(u,x,\varepsilon)) 
& =  \tfrac{1}{h} \Pi^h((u,x), \mathcal B^c(u,x,\varepsilon)) \\[10pt]
& = \tfrac{1}{h} \delta_{u+h}([\varepsilon, 1])\P_{\lfloor uh^{-1}\rfloor \rightarrow \lfloor uh^{-1}\rfloor+1}(x, \R^d \setminus \mathcal B_{\R^d}(x,\varepsilon)) \\[10pt]
& \leq \tfrac{1}{h}\sup_{x \in \R^d}\max_{0 \leq i \leq N-1}\P_{i \rightarrow i+1}(x,  \R^d \setminus \mathcal B_{\R^d}(x,\varepsilon)) \\[10pt]
&\leq \tfrac{1}{h}\sup_{x \in \R^d}\max_{0 \leq i \leq N-1} \E_{\P_{\m2i}} [ | X_{i+1} - X_i|^{4+2\, \alpha}  | X_i = x  ] \\[10pt]
& \leq  \mbox{{$h^\alpha$}} \ee^{-4-2\, \alpha}\sup_{x \in \R^d}\max_{0 \leq i \leq N-1} h \K_i^h(x)\,.
   \end{array} 
\end{equation}
{Giving (\ref{SV33})  using the definition of $ \K^h_i$ (\ref{KURT}) and 
the Kurtosis bound constraint in (\ref{SDKh}) which 
enforces the uniform boundedness of $h \K_i^h(x)$.

\subsection{Proof of  Theorem \ref{Dex} [Well posedness for the discrete problem]} 
 \label{B3} 
We give hereafter the proof of the existence of a solution to (\ref{SDKh}). This solution is unique due to the entropy, in contrast with  the continuous problem.  \\

The cost functional in (\ref{SDKh}) is (\ref{cost2}):
\[
  \I^h(\P^h)   \coloneqq
   h\, \sum_{i=0}^{N-1}   \E_{\P^h_i}  \left( F( \V_i^h(X_i^h), \D_i^h(X_i^h) )  \right) +  h \, \Ec(\P^h | \bP^h)  + \Dist(\P_0^h,\rho_0) + \Dist(\P_1^h,\rho_1)
\]
with $(\V_i^h, \D_i^h)$, the discrete drift and quadratic variation increments defined in (\ref{SMOM}).
The strict convexity  and lower semi-continuity of $\I^h$ are not immediately seen but obtained, as often in optimal transport, using a 
linear change of variable in the conditional moments $(\V_i^h(X_i^h), \D_i^h(X_i^h))$ uncovering the composition of the convex perspective function associated to $F$ and a linear operator.
Using (\ref{SMOM}), simplifying and abusing notations, for    $\beta = 1,2, ...$:
\begin{equation}
  \label{convex}
\begin{array}{ll} 
\P^h   & \rightarrow   \E_{\P^h_i} \left( F( \dfrac{1}{h}  \E_{\P^h_\tri} \left( (X^h_{i+1}-X_i^h)^\beta   \right)  )   \right) 
\mbox{ is decomposed as  }   \\[12pt]  \P^h & \rightarrow   
( \P^h_i(x_i) ,   \dfrac{1}{h} \E_{\P^h_{\m2i} (x_i,.) }   \left( (X^h_{i+1}-x_i)^\beta ) \right)   , \,\, \forall x_i    \\[12pt]
 & \rightarrow  \E_{\P^h_{i}} \left( F( \dfrac{1}{h}   \dfrac{\E_{\P^h_{\m2i}(X_i^h,.) }    \left( (X^h_{i+1}-X_i^h)^\beta  \right)}{  \P^h_{i}(X_i^h)}  )  \right).
\end{array} 
\end{equation}
where  $\P^h_{\m2i} (x_i,.) $ is to be understood as the 
measure on $\R_{t_{i+1}}$ obtained by freezing the first variable $x_i$ in the joint $\R_{t_i} \times \R_{t_{i+1}}$ probability  $\P^h_{\m2i} $. 

The convexity of additional Kurtosis constraint is also a consequence of (\ref{convex}).
The optimization problem (\ref{SDKh}) 
therefore has a unique minimizer (the relative entropy is strictly convex).

The Markovianity of the minimiser 
a consequence of the structure of $\I^h$  and the additivity properties  of the relative entropy giving  
\begin{equation}
  \label{mkine}
  \begin{array}{c} 
  \Ec(\P^h | \bP^h) \ge  \sum_{i=0}^{N-1} \E_{\P^h_i} \left(    \Ec(\P^h_\tri  | \bP^h_\tri )  \right)   \mbox{ with equality if  $\P^h$ is Markov} 
  \end{array}
\end{equation} 
(see \cite{BenamouMFG2}  lemma 3.4). 

\subsection{Proof of  Lemma \ref{Glemma} [Diffusion Coefficients rescaling]} 
 \label{B1} 
\noindent
\textbf{Scaling. }
We start with an approximation step by scaling the volatility.
Consider a solution $X(t)$ of the continuous problem with a finite cost, then introduce 
$
    X^{\alpha,\varepsilon}(t) = \sqrt\alpha X(t) + \sqrt \varepsilon B(t)\,,
$
where $B(t)$ is a $d$-dimensional Brownian motion, independent of $X$.
As a consequence, $a_t^{\alpha,\varepsilon} = \alpha a_t + \varepsilon \operatorname{Id}$.
Since  $a \in [\lambda \operatorname{Id},\Lambda \operatorname{Id}]$, it implies that 
$
 (\alpha \lambda +\varepsilon )\operatorname{Id}  \leq  a_{\alpha,\varepsilon} \leq (\alpha \Lambda +\varepsilon )\operatorname{Id} \,.
$
Now, because $\lambda < \Lambda$, and for $\delta$ small enough, a simple computation shows that taking $
\alpha = 1 - \frac{2 \,\ee}{ \lambda + \Lambda }$ and $
\ee = \delta \, \frac{\lambda + \Lambda }{\Lambda - \lambda} 
$
with $c_1, c_2>0$
there holds 
$\lambda +\delta \leq a_{\alpha,\varepsilon}  \leq \Lambda -\delta$ and 
$\sup|b_{\alpha,\varepsilon}| = \sqrt{\alpha} \sup|b| \text{ note that } \alpha <1 \,.
$
Moreover, when $\delta \to 0$, $\alpha \to 1$, $\ee \to 0$.
Now, remark that the boundary conditions at final time $1$ is lost. The distribution at time $i = 0,1$ of $Y$ is 
$
    g_{t,\varepsilon} \star [T_{\sqrt\alpha}]_\sharp(\rho(i))\,,
$
where $g_{t,\varepsilon}$ is the gaussian kernel of variance $\varepsilon t \operatorname{Id}$ and $T_{\sqrt\alpha}(x) = \sqrt\alpha x$ is a rescaling.
The same computations apply to a discrete process $X^h$, and its characteristics $a_h, b_h$.


\vspace{0.3cm}

\noindent

\textbf{Bounds on the characteristics of the discretized process. } When we consider a time continuous discretisation a continuous process satisfying the bounds on $a,b$ we might lose the bounds. However, performing first the above rescaling gives enough margin to have the discretized process satisfy the bounds.
There holds
\begin{align}\label{driftdisc}
   \E[ X^{\alpha,\varepsilon}(t_{i+1}) - X^{\alpha,\varepsilon}(t_i) \,| \, X(t_i)] 
    & = \E \left[  \int_{t_i}^{t_{i+1}}  \V_{t}^{\alpha,\varepsilon} \,dt \right]\,,
\end{align}
and 
\begin{align}\label{diffdisc}
   \E[ (X^{\alpha,\varepsilon}(t_{i+1}) - X^{\alpha,\varepsilon}(t_i))^2 \,| \, X(t_i)] &=
    \E \left[\int_{t_i}^{t_{i+1}} 2 (X^{\alpha,\varepsilon}(t)-X^{\alpha,\varepsilon}(t_i)\V_{t}^{\alpha,\ee}) + \operatorname{tr}(a_{t}^{\alpha,\ee})dt \right] \\
    & = \E \left[  \int_{t_i}^{t_{i+1}}  \operatorname{tr}(a^{\alpha,\ee}_t) \,dt \right] + C(\Lambda, B)(h^{3/2})\,.
\end{align}
Therefore, for $h$ small one can choose $
\delta_h$ to rescale the process such that the bound holds for the characteristics of the rescaled discrete process $a^{\alpha,\varepsilon,h},b^{\alpha,\varepsilon,h} $.

\textbf{Convergence of $F$. }
This follows simply by observing that
\begin{multline}\label{Fconv}
\E_Y \int_t F(a_Y,b_Y) = \E_Y \int_t G(a_Y,b_Y) \\ =\E_{X,B} \int_t G(a_Y,b_Y)
 =\E_{X,B}\int_t G(\alpha a_X+\ee,\sqrt{\alpha}b_X)
\end{multline}
since $G$ is uniformly Lipschitz in its domain, the estimate on the functional $\mathcal{I}$ follows.

\textbf{Kurtosis. }
The {\it Kurtosis} bound \eqref{KURT} is preserved whenever $\alpha<1$ and $\varepsilon$ is small enough depending on $\alpha$.   

\textbf{Entropy, continuous case. } 
In the time continuous case, the entropy term can be put into $F$, as a Lipschitz function of $a$, thus the convergence is shown as in \eqref{Fconv}.

\textbf{Entropy in the discrete case. }
    Let us first denote by $\P^h_\ee$ the law of $X + \sqrt{\ee} B$, $B$ a standard Brownian motion. 
    We first use the change of variable 
    \[
    X=(x_1,\ldots,x_N) \mapsto Z=(x_1,x_2 - x_1,\ldots,x_N - x_{N-1}).
    \]
    Note that, under this change of coordinates, $\P^h_\ee$ is the convolution of $\P^h$ with a (diagonal) Gaussian kernel $G$ defined on $\R^{dN}$ such that $\int_X \| X\|^2 G(X)dX = \varepsilon$.\footnote{Note that this choice of regularization corresponds to a fixed regularization with a Gaussian kernel on the time interval $[0,1]$ of variance $\varepsilon$.} 
    We first have that    \begin{equation} \nonumber
       \mathcal{H}(\P^h_\ee|\overline{\P}^h) \leq \int_{\R^{dN}}\mathcal{H}(\P^h|[T_X]_\sharp\overline{\P}^h) G(X)dX \,,
    \end{equation}
    where $T_X$ denotes the translation by $X \in \R^{dN}$.
    This inequality is obtained thanks to the convexity of the entropy in the first variable.  
    Let us denote by $\rho$ the image measure of $\P^h$ and $\rho_0$ the image measure of $\overline{\P}_h$ on $\R^{dN}$,
    \begin{equation*}
        \mathcal{H}(\P^h|[T_X]_\sharp\overline{\P}^h) = \int_z \rho(z)\log{\rho(z)/\rho_0(z + X)} dz\,.
    \end{equation*}
    Using the fact that $\rho_0$ is Gaussian, we have $\rho_0(z + X) = \exp{-2N\langle X ,z\rangle + N\| X \|^2} \rho_0(z)$.
    We get 
    \begin{equation*}
        \mathcal{H}(\P^h|[T_X]_\sharp\overline{\P}^h) = \mathcal{H}(\P^h|\overline{\P}^h) + \int_z \rho(z)(-2N\langle X,z\rangle + N\|X\|^2) dz\,.
    \end{equation*}
    Integrating over $X$ with respect to a the centered variable gives $\int_X \int_z \rho(z)(-2N\langle X,z\rangle dz G(X)dX= 0$. This implies
    \begin{equation*}
        \int_X\mathcal{H}(\P^h|[T_X]_\sharp\overline{\P}^h)G(X)dX = \mathcal{H}(\P^h|\overline{\P}^h) +\int_X  N\|X\|^2 G(X)dX\,.
    \end{equation*}
Thus, we get
     $h\mathcal{H}(\P^h_\ee|\overline{\P}^h) \leq  h\mathcal{H}(\P^h|\overline{\P}^h) + \ee\,.
    $
    We now treat similarly the scaling in $\alpha$, 
    \begin{multline}
        \mathcal{H}([S_{\sqrt{\alpha}}]_\sharp \P^h|\overline{\P}^h) = \int_z \rho(z)\log{\alpha^{-dN/2}\rho(z)/\rho_0(\sqrt{\alpha}z)} dz \\
        =-dN/2\log{\alpha} - \int_z \rho(z) \log{\rho_0(\sqrt{\alpha}z)}dz + \int_z \rho(z) \log{\rho(z)}dz \\
        = -dN\log{\alpha}/2 +\mathcal{H}(\P^h|[\overline{\P}^h) - \int_z \rho(z) 2N(\alpha - 1) \|z\|^2 dz\,. 
    \end{multline}
We obtain
\begin{equation} \nonumber
   h \mathcal{H}(\P^h|[S_{\sqrt{\alpha}}]_\sharp\overline{\P}^h) = -d\log{\alpha}/2 + h\mathcal{H}(\P^h|[\overline{\P}^h) - \int_z \rho(z) 2(\alpha - 1) \|z\|^2 dz\,. 
\end{equation}
The quantity $\int_z  \|z\|^2 \rho(z) dz$ is equal to the sum of the second moment of $\rho(t=0)$ and the quadratic variation of $\P^h$ which is finite uniformly in $h$.
This completes the proof the Lemma.

\begin{remark}
Another possibility for this proof is to apply the rescaling and regularisation also to the reference measure. Then $ \Hc( \P^h_\ee | \bP^h_\ee) \le \Hc( \P^h  | \bP^h ) $ is straightforward.
We then have to deal a specific relative entropy with a reference diffusion which depends on $\ee$ via $\bD_\ee = \alpha \bD + \sigma $ ($\alpha$ and  $\sigma$ depend on $\ee$).
 Then (\ref{F11}) becomes 
$
    \I^0_\ee( \tP^0_\ee)  \le   \liminf_{h\searrow 0}  \I^h( \P^h)  + O(\ee) 
$
where $\I^0_\ee = \I^0 +  R_\ee $ with $  R_\ee = \E_{\tP^0_\ee} ( \S^\Ic ( \tD^0_\ee | \bD_\ee) - \S^\Ic ( \tD^0_\ee | \bD))  = O(\bD_\ee - \bD)  = O(\ee) $.
\end{remark}

\subsection{Proofs of  Lemma \ref{ThRegInContinuousTime}  [Regularization time continuous case]} 
 \label{B21} 
The first step concerns space regularization and consists of adding a Gaussian variable of variance $\sigma$ to the process $X(t)$. The corresponding  density at time $0$  is given by $g_{\sigma} \star \rho(0)$ where $\rho(0)$ is the marginal at time $0$ of $X(t)$. 

The second step involves extrapolating the process $X$ in time with a process that has a finite cost. 
Choose $\sigma_0 \in [\lambda,\Lambda]$ and define on $[-2\varepsilon,0]$, the solution of the heat equation starting at time $ t = -2\varepsilon$ with the initial condition $g_{\sigma - (2\varepsilon \sigma_0)} \star \rho(0)$. 
Here we need to assume $2 \varepsilon \sigma_0 < \sigma$ which is always satisfied for $\varepsilon$ small enough.
At time $t = 0$, the density equals $g_\sigma \star \rho(0)$.
For the times $t > 1$, we also use a diffusion with coefficient $\sigma_0$ so that $X(t)$ can be extended to the time interval $[-2\varepsilon, 1 + 2\varepsilon]$ with a finite cost outside the interval $[0,1]$ of order $O(\varepsilon)$.

Let $k: \R \to \R_+$ be a smooth nonnegative function with support in the unit ball, such that $\int k(y) dy = 1$.
We denote the kernel $\eta_\varepsilon(t,x) =  k(t/\varepsilon)/\varepsilon$ with support in the time variable is thus contained in the ball of radius $\varepsilon$. 

We consider
\begin{equation}
\P^0_{t,\ee} := (\P_t^0 * \etae) \quad \quad 
\V^0_{t,\ee} = \dfrac{ ( \V_t^0 \, d\P^0_{t}) * \etae   }{ d \P^0_{t,\ee}} \quad \quad 
\D^0_{t,\ee}  = \dfrac{ ( \D_t^0 \, d\P^0_{t}) * \etae   }{ d \P^0_{t,\ee}} \,\cdot
\end{equation}
By convexity, the drift and diffusion coefficients satisfy the hard constraints of the lemma's hypotheses.  Moreover, these coefficients (drift and diffusion) are smooth in time (due to the convolution with $\eta_\varepsilon$) and space (due to the initial regularization in the first step) so that by standard arguments, there exists a unique Markov process solving the corresponding diffusion equation which is well defined on the time interval $[-\varepsilon,1+\varepsilon]$. We denote such a process by $X^\varepsilon(t)$.
By the control on the time extrapolation and the convexity of the functional $\mathcal{I}^0$, we have
\begin{equation} \nonumber 
   \int_{-\varepsilon}^{1 + \varepsilon} \int_x   G(\V_{t,\varepsilon}(x),\D_{t,\varepsilon}(x)) \,  d\rho_{t,\varepsilon}(x)  \, dt \leq  \mathcal{I}^0(\rho^0_{t},\V^0_{t},\D_{t}^0) + O(\varepsilon) \,.
\end{equation}
We rescale the time interval from $[-\varepsilon,1+\varepsilon]$ to $[0,1]$ to control the soft constraint on the boundary terms.
Due to the hypotheses (hard constraints on the drift and diffusion coefficients), $Y \coloneqq X^\varepsilon((1 + 2\varepsilon)(t+\varepsilon))$ satisfies the hard constraints on the drift $|b_{t,\varepsilon}| \leq B$ and volatility $\lambda \leq a_{t,\varepsilon} \operatorname{Id} \leq \Lambda$.
Importantly, the distributions of $Y(0)$ and $Y(1)$ are (infinite) mixture of Gaussian convolutions of $\rho(0)$ and $\rho(1)$ the marginals of the process $X(t)$, since the kernel has compact support in time in $[-\varepsilon,\varepsilon]$ and the extension in time is by heat diffusion. As a consequence and from the hypotheses on $\Dist$, one has for $i = 0,1$, $|\Dist(\P^Y_{t = i},\rho_i) - \Dist(\P^X_{t = i},\rho_i)| \leq O(\varepsilon) + O(\sigma)$ (where we recall that $\rho_i$ are the (soft) boundary values). 
Thus, we obtain the result $\mathcal{I}(\P^Y) \leq \mathcal{I}(\P^X) + O(\varepsilon) + O(\sigma)$.

\vspace{0.3cm}

\subsection{Proof of  Lemma \ref{LF3}}  
 \label{B4}

For the specific entropy, Proposition \ref{p1} $i)$ gives for the entropic part of $\I(\P^h_\varepsilon)$: 
\begin{equation}
\label{ee1}
    \lim_{h\searrow 0 } h\, \Ec(\P^h_\ee | \bP^h) = \S^{\Ec}(\P^0_\ee | \bP).
\end{equation}
For the integrand part, we first start with
(already in the proof of lemma \ref{Glemma} $iv)$) 
\begin{align*}
   \E[ X(t_{i+1}) - X(t_i) \,| \, X(t_i)] &=
    \E_{\P^h_{\tri,\ee}} \left[  \int_{t_i}^{t_{i+1}}  \V_{t,\ee}(X_t) \,dt  + \int_{t_i}^{t_{i+1}} \sqrt{\D_{t,\ee}(X_t)} \,  dB_t  \right]\\
    & = \E_{\P^h_{\tri,\ee}} \left[  \int_{t_i}^{t_{i+1}}  \V_{t,\ee}(X_t) \,dt \right]\,,\\
   \E[ (X(t_{i+1}) - X(t_i))^2 \,| \, X(t_i)] &=
    \E_{\P^h_{\tri,\ee}}\left[\int_{t_i}^{t_{i+1}} 2 (X(t)-X^h_i)\V_{t,\ee}(X(t)) + \operatorname{tr}(a_{t,\varepsilon}(X_t) dt \right] \\
    & = \E_{\P^h_{\tri,\ee}} \left[  \int_{t_i}^{t_{i+1}}  \operatorname{tr}(a_{t,\varepsilon}(X_t) \,dt \right] + O(h^{1 + \alpha})\,,
\end{align*}
where $\alpha > 0$ depends on the integrability of $b$. 
This formula implies that the  discrete characteristic coefficients $\V^h_i$  and $\D^h_i$ for $\P^h_\ee$
 satisfy for $h$ small enough 
$\lambda Id < \D < \Lambda Id $ and 
$| \V| < B$. 
We also have:
\begin{multline}
    h \, \E_{\P^h_{i,\ee}} [ F\left(  \V_{i,\ee}^h(X(t_i)),\D_{i,\ee}^h(X(t_i)))\right]  =\\ h \, \E_{\P^h_{i,\ee}} \left[ G\left(  \frac{1}{h} \E_{\P^h_{\tri,\ee}} \left[  \int_{t_i}^{t_{i+1}}  \V_{t,\ee}(X(t)) \,dt \right],\frac{1}{h}\E_{\P^h_{\tri,\ee}}  \left[  \int_{t_i}^{t_{i+1}}  \operatorname{tr}(a_{t,\varepsilon}(X(t)))+ O(h^\alpha) \,dt \right] \right) \right]\,.
\end{multline}
Using Jensen's inequality and the Lipschitz assumption on $G$, we get
\begin{equation*} 
    h \, \E_{\P^h_{i,\ee}} [ G(  \V_{i,\ee}^h(X(t_i)),\D_{i,\ee}^h(X(t_i)))] \leq 
    \int_{t_i}^{t_{i+1}}\E_{\P^h_{i,\ee}}G(b_{t,\varepsilon}(X(t),\operatorname{tr}(a_{t,\varepsilon}(X(t))))]dt + O(h^{\alpha + 1})\,.
\end{equation*}
The control on the time boundary marginal terms follows from lemma \ref{Glemma} and \ref{ThRegInContinuousTime}.

\subsection{Proof of  Lemma  \ref{FXlemma} [Regularization time discrete case]} 
 \label{B22}

\textbf{Tightness: }
We need tightness of the quantities $\P^h$, $m^h$ and $n^h$ defined by
\begin{align}\label{EqMoments}
   m_i^h = \frac{1}{h}(\int_{x_{i+1}}  (x_{i+1} - x_i)\P^h_{\tri}(x_{i+1}|x_i)) \P^h_i(x_i)\,,
\\
   n_i^h = \frac{1}{h}(\int_{x_{i+1}}  (x_{i+1} - x_i)^2\P^h_{\tri}(x_{i+1}|x_i)) \P^h_i(x_i)\,.
\end{align}
These quantities are discrete in time; we will define this quantity continuously to obtain tightness in the space of time-space measures.
By considering either the Markov chain or the time-linearly interpolated process, we can consider $\P^h$ either as discrete probabilities or as probabilities on the set of continuous paths.
Let us define more explicitly the time-dependent path of marginals
$\rho^h(t) = 
\sum_{i=0}^N   \mathbf 1_{[t_i,t_{i+1})} \P_{i}^h\,.
$
To obtain tightness of the $\rho^h$, we use linear interpolation of $\P^h$ denoted $\tilde \P^h$.
We start with the tightness of $\P^h$. Recall that we have a control on the $\alpha$-moments ($\alpha > 4$) of $\P^h$. In particular, the linearly interpolated process $X^h(t)$ satisfies, for $s,t \in [0,T], |s-t| < 1/N$, for some $c>0$,
$
    \E[\| X_s^h - X_t^h\|^\alpha]  \leq c |s-t|^{\alpha/2 - 1}
$.
 The Kolmogorov lemma implies that the sequence $\tilde \P^h$ only charges H\"older trajectories with H\"older coefficient $1/2 - 2/\alpha - \varepsilon' >0$ (positive since $\alpha > 4$) for all positive $\varepsilon'$ sufficiently small. 
 Hence the sequence $\tilde \P^h$ is tight and consequently, so is $\rho^h(t)$.

To apply the criterion in \cite{SVbook}, we need to check if the "jump condition" is satisfied on $\P^h$. It is guaranteed by the kurtosis bound. 
We now deal with quantities $m_i^h$ and $n_i^h$.
We have $ m^h = \sum_{i=0}^N   \mathbf 1_{[t_i,t_{i+1})} m_i^h \,,\quad
    n^h = \sum_{i=0}^N  \mathbf 1_{[t_i,t_{i+1})} n_i^h \,.
$
We want to prove that the measures (defined on $[0,1] \times \R^d$) $m^h$ and $n^h$ are tight. For that, we apply the De la Vallée Poussin lemma. We propose to bound the $1<1 + \alpha < 2$ moment, using the inequality $ab \leq \frac 12 (a^2 + b^2)$,
\begin{equation}
    \int_0^1 \int_{\R^d} | x|^{1 + \alpha} n^h(t,x) dt \leq
    h \left(\sum_{i = 1}^N \frac{1}{2} \int | x_i|^{2 + 2\alpha} \P_i^h + \frac{1}{2h^2} \int | x_{i+1} - x_i|^{4} \P^h \right)\,.
\end{equation}
The second term on the right-hand side is bounded due to the Kurtosis bound. It suffices to prove a uniform bound in time of $(2+2\alpha)$-moment to bound the first term, which is bounded by the $4$-moment when $\alpha < 1$ which is true.
Using a telescopic sum, there exists a positive constant $M>0$ such that
\begin{equation}
 \int_0^1 \int_{\R^d} | x_i|^{4} \P^h_{i,i+1}(x_i,x_{i+1}) dt \leq M \sum_{i = 1}^N  (\int |x_0|^{4} + \sum_{j = 1}^i | x_{i+1} - x_i|^{4} \P^h) \,.
\end{equation}
Our hypothesis bounds the first term on the right-hand side on $\rho_0$. The second term is bounded by $O(h)$ due to the kurtosis bound.
The proof is similar for the quantity $m^h$. 

\noindent
\textbf{Scaling the coefficients and space regularization: }
To prepare for the time extrapolation below, we transform $\P^h$ by using Lemma \ref{Glemma}. Denoting $X^h(t)$ any interpolation of $\P^h$, we consider the transformation of the type (for well-chosen parameters $\alpha,\delta$, see Lemma \ref{Glemma})
\begin{equation}
    Y(t) = \sqrt\alpha X^h(t) + \sqrt \delta B(t) + B_0\,,
\end{equation}
with $B(t)$ a standard Brownian motion and $B_0$ a centered Gaussian variable of variance $\varepsilon \operatorname{Id}$. Lemma \ref{Glemma} gives that $Y$ satisfies the estimates on the functional and the hard constraints on the drift and the volatility. The kurtosis bound is also trivially satisfied.

\noindent
\textbf{Time extrapolation: }
We want to again define the convolution in time and space of the probability $\P^h$ extended on the path space by linear interpolation of the curves. However, it is not defined outside the time interval $[0,1]$, which is needed to regularize the process in time and control the density values at times $0,1$. We use the two first steps of the proof of Lemma \ref{ThRegInContinuousTime} to construct an extrapolation on the time interval $[-2\varepsilon,0]$ and $[1,1+2\varepsilon]$ with finite cost $\mathcal I^0(\P) + O(\varepsilon) + O(\delta)$. 
  We now discretize it in time as in the Gamma-Limsup proof so that the cost $\mathcal{I}^h$ on $[-2\varepsilon,0]$ and $[1,1+2\varepsilon]$ is of the order $O(\varepsilon) + O(h)$. 
  This discretization gives us a  time probability measure in discrete time still denoted by $\P^h$. 

\noindent
\textbf{Time regularization: }
We use $\eta_\varepsilon(s,x) \coloneqq k(s/\varepsilon)/\varepsilon$ as in Lemma \ref{ThRegInContinuousTime}.
We now introduce a time-dependent curve with values in the convex set of plans, i.e. probability measures on the product space,
\begin{equation}\Pi^h(t,x,y) = 
\sum_{i} \mathbf 1_{[t_i,t_{i+1})} \P_{i,i+1}^h (x,y) \,,
\end{equation}
which is a collection of time-dependent plans indexed by $t \in [-2\varepsilon,1 + 2 \varepsilon]$. 
To a given plan $\pi(x,y) \in \mathcal{P}(\R^d \times \R^d)$, one can define its conditional moments (exactly as in Formula \eqref{EqMoments}), in probabilistic notations
$
    m^h(\pi) \coloneqq \frac{1}{h}\int_{y}  (y - x)\pi(x,dy)\,,
$
and 
$
    n^h(\pi) \coloneqq \frac{1}{h}\int_{y}  (y - x)^2\pi(x,dy)\,.
$
Now, we observe that  $ \nonumber
m^h(t) = m^h(\Pi(t)) \,,\quad n^h(t) = n^h(\Pi(t))\,.
$
In addition, the marginalization of $\Pi$ on the first variable gives $\rho^h(t)$.
Note also that the following property holds by definition
$
    \forall t \quad [p_2]_\sharp\Pi(t) = [p_1]_\sharp\Pi(t+h)\,,
$
where the $[p_i]_\sharp, i = 1,2$ is the pushforward operator on the first and second variables.
We now regularize in time this object
$
    \Pi_\varepsilon^h = \eta_\varepsilon \star \Pi^h \,.
$
\noindent
Importantly, the regularization on $\Pi^h$ induces a regularization on the quantities $\rho, m, n$:
By linearity, we have that 
\begin{equation} \nonumber
 [\pi_1]_\sharp\Pi^h_\varepsilon = \eta_\varepsilon \star \rho^h  \quad m^h(\Pi^h_\varepsilon) = \eta_\varepsilon \star m^h \quad n^h(\Pi^h_\varepsilon) = \eta_\varepsilon \star n^h \,.
\end{equation}
In addition, it defines a discrete Markov chain:
By linearity, we have the property:
$\forall t \quad [p_2]_\sharp\Pi_\varepsilon(t) = [p_1]_\sharp\Pi_\varepsilon(t+h)$.
From $\Pi_\varepsilon$ we construct a probability on the path space: we  evaluate $\Pi_\varepsilon(ih)$ which is a plan and we note that 
$
    [p_2]_\sharp\Pi_\varepsilon(ih) = [p_1]_\sharp\Pi_\varepsilon((i+1)h)\,,
$
by the property mentioned above.
By the standard gluing lemma \cite{AGS}, we construct a corresponding Markov chain using the joint probabilities $\Pi_\varepsilon(ih)$ that we denote $\P^h_{\de}$.
By convexity of the integral part of the functional, we estimate the cost 
\begin{equation}
    \label{b1} 
\I^h( \P^h_{\de})
\leq \I^h( \P^h)  + O(\varepsilon) + O(h) + O(\delta)\,,
\end{equation}
where the term $O(\varepsilon) + O(h)$ is due to the extension in time of $\P^h$ and its time discretization. 
On $\rho^h_\de,m_\de^h, n_\de^h$ which are smoothing of a tight sequence of measures, we can extract a subsequence that is bounded in $L^\infty_{t,x}$ and smooth in both time and space, uniformly converging on compact subsets.

\noindent
\textbf{Time-rescaling: }
Similarly to Lemma \ref{ThRegInContinuousTime},
we rescale time so that the quantities are all defined on the $[0,1]$. The cost of this rescaling is $O(\varepsilon)$. Note that by rescaling, the discretization timestep is also rescaled, by abuse of notation, we still keep the same letter $h$. More importantly, the rescaling increases the volatility linearly in $\varepsilon$ and we need to check that the hard constraints on the volatility are preserved. This is indeed the case since Lemma \ref{Glemma} has been applied beforehand.

\noindent
\textbf{Riemann integral: }
Recall the functional $\mathcal{J}$ as having three arguments
\begin{equation} \nonumber
    \mathcal{J}(\rho,m,n) = \int_{t,x} F(m/\rho,n/\rho) + \S^{\Ic}_\bD(n/\rho)   \, d\rho(t,x) dt \,,
\end{equation}
where $\S^{\Ic}_\bD$ denotes the specific entropy with reference diffusion $\bD$ using $n,\rho$ as arguments.
Note that we want to study $\mathcal{J}(\rho(\Pi(t)),m^h(\Pi(t)),n^h(\Pi(t)))$.
Then, the discrete formulation $\I^h$ (as well as the specific entropy) can be viewed as a Riemann integral in time:
\begin{multline}
    \I^{h}(\P^h) \geq \sum_{i = 0}^N h \int_{x} F(m^h(\Pi(ih))/\rho(\Pi(ih)),n^h(\Pi(ih))/\rho(\Pi(ih))) \\ + \S^{\Ic}_\bD(n^h(\Pi(ih))/\rho(\Pi(ih)))d\rho(\Pi(ih))\,,
\end{multline}
where $\rho(\Pi(ih))$ is  the marginal of $\Pi(ih)$ on the first variable.  
In addition, we will use below that this Riemann sum coincides with the functional $\mathcal{J}$ evaluated on $(\tro,\tm,\tn)$.
By regularization, $(m_\de^h(t),n_\de^h(t),\rho^h_\de(t))$ is smooth in time. It is also the case for the function integrated in space.
To prove it, we use that the drift and volatility coefficients are bounded in $L^\infty$.
We have
\begin{multline}
    |\int_x F(m_\de^h(t)/\rho_\de(t),n_\de^h(t)/\rho_\de(t)) + \S^{\Ic}_\bD(m_\de^h(t)/\rho_\de(t),n_\de^h(t)/\rho_\de(t)) d\rho_\de(t) \\ - F(m_\de^h(t')/\rho_\de(t'),n_\de^h(t')/\rho_\de(t')) + \S^{\Ic}_\bD(m_\de^h(t')/\rho_\de(t'),n_\de^h(t')/\rho_\de(t')) d\rho_\de(t')| \\ \leq M \int_x |\rho_\de(t') - \rho_\de(t))| \leq M C(\varepsilon)|t-t'|\,,
\end{multline}
where $M$ bounds the integrand (which is possible since all the terms are bounded independently of $h,\varepsilon,\delta$) and $C(\varepsilon)$ is a constant coming from the convolution kernel and depending on $\varepsilon$. Therefore, we get
\begin{equation}\nonumber
   \I^{h}(\P_\de^h)\geq \mathcal{J}(\tro_\de^h,\tm_\de^h,\tn_\de^h) \geq  \mathcal{J}(\rho_\de^h,m_\de^h,n_\de^h) + O(h)\,,
\end{equation}
which gives the inequality \eqref{IneqRiemannSum}.
Note that the $O(h)$ is due to the difference between the Riemann sum and the integral.   
By lower-semicontinuity of $\mathcal{J}$, one has 
$
    \mathcal{J}(\rho_\de,m_\de,n_\varepsilon) \leq \lim_{h\to 0} \mathcal{J}(\rho_\de^h,m_\de^h,n_\de^h)\,.
$
In particular, the two previous inequalities  imply $iii)$ of the lemma and  
$
    \mathcal{J}(\rho_\de,m_\de,n_\de) \leq \lim_{h\to 0} \I^h(\P^h_\de)\,.
$

\noindent
\textbf{Conclusion: }
The coefficients $m_\de^h,n_\de^h$, and $\rho^h_\de$ converge uniformly in space and time:
It implies that $b_h(ih) = \frac{m_\de^h(ih)}{\rho^h_\de(ih)}$ and 
$a_h(ih) = \frac{n_\de^h(ih)}{\rho^h_\de(ih)}$ 
converge uniformly on every compact in time and space since $\rho_\varepsilon$ is bounded below by a positive constant on every compact set.
Thus, we can apply Theorem \ref{ThStroockVaradhanExtendedInTime} to obtain $ii)$ of the lemma: the sequence $\P_\varepsilon^h$ uniformly converges to $\P_\varepsilon$ the probability on the path space of a diffusion process. By uniform convergence, the drift and diffusion coefficients can be identified by $b_\de(t,x) = \frac{m_\de(t,x)}{\rho^h_\de(t,x)}$ and
$a_\de(t,x) = \frac{n_\de(t,x)}{\rho^h_\de(t,x)}$.
Finally, we get 
\begin{equation} \nonumber
 \mathcal{J}(\rho_\de,m_\de,n_\de) \leq \lim\inf_{h \to 0}\I^h(\P^h) + O(\varepsilon) +O(\delta)\,.
\end{equation}

\subsection{Additional  Lemmas} 
 \label{A5} 

{
\begin{lemma}
\label{lemmaMH}

The quantities $\V_t$, $\D_t$ are the characteristic coefficients of the semi-martingale $X$ with law $\P \in \Pc^1$. 
For every $p \geq 1$ and $q,q' > 1$, we have, in full generality:
\begin{align}
\mathbb E\big[|X_t-X_s|^p\big]  & \leq C_p \Big((t-s)^{p(1-\frac{1}{q})}\big(\mathbb E[\int_0^1|b_u(X_u)|^{\max(p,q)}du]\big)^{\min(p/q, 1)} \nonumber \\
& +(t-s)^{\frac{p}{2}(1-\frac{1}{q'})}\big(\mathbb E[\int_0^1|a_u(X_u)|^{\max(p/2,q')}du]\big)^{\min(p/2q'), 1)}\Big),
 \end{align}
where $C_p>0$ only depends on $p$. 
\end{lemma}
\begin{proof}[Proof of the lemma]
In the following, the notation $C_p$ denotes a positive number that only depends on $p$ and that may vary at each occurence.
We start with the following observation: for any measurable random process $f_u$, we have, for any $q>1$, by H\"older's inequality
\begin{align}\mathbb E\big[|\int_s^tf_udu|^p\big] & \leq (t-s)^{p(1-\frac{1}{q})}
\left\{
\begin{array}{lll}
\mathbb E[\int_0^1|f_u|^pdu] & \mathrm{if} & p \geq q \nonumber \\ \\
\mathbb E[\int_0^1|f_u|^qdu]^{p/q} & \mathrm{if} & p \leq q 
\end{array}
\right. \nonumber \\
&=  (t-s)^{p(1-\frac{1}{q})} (\mathbb E[\int_0^1|f_u|^{\max(p,q)}du])^{\min(1,p/q)} \label{eq: holder opt}
\end{align}
Next, from
$X_t-X_s = \int_s^t b_u(X_u)du+\int_s^ta_u(X_u)^{1/2}dB_u$
and thus
$\mathbb E[|X_t-X_s|^p] \leq C_p(\mathbb E[|\int_s^t b_u(X_u)du|^p]+\mathbb E[|\int_s^ta_u(X_u)^{1/2}dB_u|^p])$
since $p\geq 1$, we successively have
$\mathbb E[|\int_s^t b_u(X_u)du|^p] \leq (t-s)^{p(1-\frac{1}{q})} (\mathbb E[\int_0^1|b_u(X_u)|^{\max(p,q)}du])^{\min(1,p/q)}$
and
\begin{multline}
 \mathbb E[|\int_s^t a_u(X_u)^{1/2}du|^p]   \leq C_p \mathbb E[|\int_s^t a_u(X_u)du|^{p/2}]
 \\ \leq (t-s)^{\frac{p}{2}(1-\frac{1}{q'})} (\mathbb E[\int_0^1|a_u(X_u)|^{\max(p/2,q')}du])^{\min(1,p/2q')}
\end{multline}
by \eqref{eq: holder opt} and the Burkholder-Davis-Gundy inequality. The conclusion follows.
\end{proof}
}

A brief excerpt from chapter 9 of \cite{AGS} and 
\cite{Trevisan1,Trevisan16}:

\begin{definition}[perspectives functions over measures a.k.a. entropies]
\label{defPF}
Assuming  that $ F(.) = \theta (|.|) :\R^d \rightarrow [0,\infty]$
 is  a convex lower semi-continuous function 
with superlinear growth at infinity.
The perspective function (aka the entropy) of $F$ evaluated at the measures 
$(M,\nu)$ is defined as 
\[
(\nu,M) \rightarrow \Fc(  M  | \nu) = \left\{ 
\begin{array}{l}
 \E_{\nu}(F(\dfrac{dM}{d\nu}))   \mbox{  if $M \ll \nu$.}  \\
     +\infty \mbox{ else.}
     \end{array}  \right.
\]
The classic example is the Shanon relative entropy $F(.)  = -\log{ . | {\bar M}}$.  The ``nice'' properties of $\Fc$ are 
inherited from its dual formulation (\ref{dualEntropy}), 
\end{definition}

\begin{definition}[Weak Fokker-Planck Solutions]
\label{defFP}
We will refer to this property as:  ``$(\nu_t,\V_t,\D_t) \in FP(\rho_{0}, \rho_1$''). \\

A  curve $(\nu_t)_{t \in [0,1]}$ of probability measures is  a weak 
solution of the Fokker-Planck with drift coefficients $(\V_t)$ and diffusion 
coefficients $(\D_t)$  if  $(\V_t,\D_t) \in L^1_{t,x}(\nu_t)$ and for all $f \in C_b^{1,2}((0,1) \times \R^d)$
\[
\int_0^1 \int [ \partial_t f+ \V_t \, \partial_x f + \D_t \,  \partial_{xx} f ] \,  d\nu_t(x) \, dt = 0 
\]
In our case we add the initial-final marginal conditions : 
$
\nu_{0,1}  = \rho_{0,1}
$.
It is not restrictive to assume that $(\nu_t)$ is  narrowly continuous. 
For more details and $d>1$ see \cite{Trevisan16} definition 2.2 and remark 2.3.
\end{definition}

\begin{proposition}
\label{pE}

\begin{itemize}
\item[i)] $\Fc$ in definition \ref{defPF} is jointly convex and lower semi continuous 
in $(M,\nu)$. 
\item[ii)] 
$(M^h,\nu^h)$ two sequences of Borel positive measure in $\R^d$, such that, $\nu^h$ weakly converges to $\nu^0$, 
$M^h$ is a.c. w.r.t. $\nu_h$ for all $h$ and 
$
\sup_h \E_{\nu^h}(F( \dfrac{dM^h}{d\nu^h})) < \infty  .
$
Then $M^0$ is a.c. w.r.t. $\nu^0$ and 
$
\lim \inf_{h\searrow0} \E_{\nu^h}(F(\dfrac{dM^h}{d\nu^h})) \ge  \E_{\nu^0}(F(\dfrac{dM^0}{d\nu^0}))
$.
\item[iii)] Let  $\etae$ be a smooth regularisation kernel with $k$ bounded derivatives. 
Denote $M_\ee = M * \etae$, $\nu_\ee = \nu * \etae$. Then, 
$
\Fc( M_\ee | \nu_\ee ) \le \Fc( M | \nu )
$
\item[iv)] If $(\nu_t)$ is a weak solution of $FP_{\rho_{0,1}}(\V_t,\D_t)$ ,  $(\V_t,\D_t) \in \Lc^p(\nu_t)$   $(\nu_{t,\ee})$  is 
a weak solution of $FP_{\rho_{0,1,\ee}}(\V_{t,\ee},\D_{t,\ee})$ where for all $t$ 
$(\nu_{t,\ee} := (\nu_t * \etae)$ and we mollify the ``moments'': 
\[\
\V_{t,\ee} = \dfrac{( \V_t \, d\nu_t) * \etae   }{ d \nu_{t,\ee}} \quad \quad 
\D_{t,\ee}  = \dfrac{( \D_t \, d\nu_t) * \etae   }{ d \nu_{t,\ee}}. 
\]
Then, $(\V_{t,\ee},\D_{t,\ee})$ are in $\Lc^p(\nu_{t,\ee})$ and well defined uniformly bounded $C^k$ densities:
$
\sup_{t \in [0,1]} \| (\V_{t\ee},\D_{t,\ee})\|_{C^k_b(B)}  < +\infty 
$
for all bounded set $B\in \R^d$.  
\item[v)] Applying $iii)$ to the setting in $iv)$ we get 
$
\E_{\nu_{t,\ee}}(F( \V_{t,\ee} )) \le 
\E_{\nu_{t}}(F( \V_{t})) \,.
$
\end{itemize} 
\end{proposition}
\begin{proof} 
These results are a direct application of the dual form of the entropy given in lemma 9.4.3 and 9.4.4 \cite{AGS}, see also lemma A.1 \cite{Trevisan16}. 
\end{proof}

\end{document}